\documentclass[preprint,sort&compress,12pt]{elsarticle}
\usepackage{amscd,amssymb,amsmath,amsthm}
\usepackage[all]{xy}
\usepackage{array}
\usepackage{etoolbox}
\usepackage{mathtools}
\DeclarePairedDelimiter{\ceil}{\lceil}{\rceil}
\DeclarePairedDelimiter{\floor}{\lfloor}{\rfloor}
\makeatletter
\patchcmd{\ps@pprintTitle}{\footnotesize\itshape
       Preprint submitted to \ifx\@journal\@empty Elsevier
       \else\@journal\fi\hfill\today}{\relax}{}{}
\makeatother

\newdir{ >}{!/8pt/\dir{}*\dir{>}}

\newtheorem{theorem}{Theorem}[section]

\newtheorem{lemma}[theorem]{Lemma}

\newtheorem{prop}[theorem]{Proposition}

\newtheorem*{con7*}{Conjecture 7*}

\theoremstyle{definition}
\newtheorem{definition}[theorem]{Definition}

\newtheorem*{th315}{Theorem 3.15}

\newtheorem*{th46}{Theorem 4.6}
\newtheorem*{th49}{Theorem 4.9}
\newtheorem*{th410}{Theorem 4.10}
\newtheorem*{th63}{Theorem 6.3}
\newtheorem*{th56}{Theorem 5.6}
\newtheorem*{th53}{Theorem 5.3}
\theoremstyle{remark}

\numberwithin{equation}{theorem}

\DeclareMathOperator{\M}{M(G)}
\DeclareMathOperator{\m}{M}
\DeclareMathOperator{\e}{exp}

\journal{}

\begin{document}

\begin{frontmatter}

\title{ On the Exponent of the Schur multiplier.}

 \author[IISER TVM]{A. Antony}
\ead{ammu13@iisertvm.ac.in}
\author[IISER TVM]{K. Patali}
\ead{patalik16@iisertvm.ac.in}
\author[IISER TVM]{V.Z. Thomas\corref{cor1}}
\address[IISER TVM]{School of Mathematics,  Indian Institute of Science Education and Research Thiruvananthapuram,\\695551
Kerala, India.}
\ead{vthomas@iisertvm.ac.in}
\cortext[cor1]{Corresponding author. \emph{Phone number}: +91 8078020899.}

\begin{abstract}
A longstanding problem attributed to I. Schur says that for a finite group $G$, the exponent of the second homology group $H_2(G, \mathbb{Z})$ divides the exponent of $G$. In this paper, we prove this conjecture for finite nilpotent groups of odd exponent and of nilpotency class 5, $p$-central metabelian $p$ groups, and groups considered by L. E . Wilson in \cite{LEW}. Moreover, we improve several bounds given by various authors.

\end{abstract}

\begin{keyword}
 Schur Multiplier \sep regular $p$-groups \sep powerful $p$-groups \sep potent $p$ groups  \sep covering group \sep group actions.
\MSC[2010]   20B05 \sep 20D10 \sep 20D15 \sep 20F05 \sep 20F14 \sep 20F18 \sep 20G10 \sep 20J05 \sep 20J06 
\end{keyword}

\end{frontmatter}

\section{Introduction}

The Schur multiplier of a group $G$, denoted by $\M$ is the second homology group of $G$ with coefficients in $\mathbb{Z}$, i.e $\M=H_2(G, \mathbb{Z})$. A longstanding conjecture attributed to I. Schur says that 
\begin{equation}\label{E1}
\tag{1}  \e(\M)|\e(G). 
\end{equation}

To prove ($\ref{E1}$), it is enough to restrict ourselves to $p$ groups using a standard argument given in Theorem 4, Chapter IX of \cite{JPS}. A. Lubotzky and A. Mann showed that (\ref{E1}) holds for powerful $p$ groups(\cite{L.M}), M. R. Jones in \cite{MRJ} proved that (\ref{E1}) holds for groups of class 2,  P. Moravec showed that (\ref{E1}) holds for groups of nilpotency class at most 3, odd order class 4, potent $p$ groups, metabelian $p$ groups of exponent $p$, $p$ groups of class at most $p-2$ (\cite{P.M.1}, \cite{P.M.2}, \cite{P.M.3}) and some other classes of groups. The general validity of $(\ref{E1})$ was disproved by A. J. Bayes, J. Kautsky and J. W. Wamsley in \cite{BKW}. Their counterexample involved a 2 group of order $2^{68}$ with $\e(G)=4$ and $\e(\M)=8$. Nevertheless for finite groups of odd exponent, this problem of Schur remains open till date. This problem has remained open even for finite $p$ groups of class 5 having odd exponent. The purpose of this paper is to prove (\ref{E1}) for finite $p$ groups of class 5 with odd exponent and to also prove the above mentioned results of \cite{L.M}, \cite{P.M.1}, \cite{P.M.2} and \cite{P.M.3} for odd primes and hence proving all these results using a common technique and bringing them under one umbrella. We briefly describe the organization of the paper by listing the main results according to their sections.

In \cite{P.M.5}, the author proves that if $G$ is a group of nilpotency class 5, then $\e(\M)\mid (\e(G))^2$. In the next Theorem, we improve this bound and in fact show the validity of (\ref{E1}). Entire section 3 is devoted to proving this theorem.

\begin{th315}
Let $G$ be a finite $p$-group of nilpotency class 5. If $p$ is odd, then $\e(\M)\mid \e(G)$.
\end{th315}

In Section 4, we prove three main theorems which we list below. In \cite{P.M.3}, P. Moravec shows that $(\ref{E1})$ holds for $p$ groups of class less than $p-1$. Authors of \cite{H.M.M} prove the same for class less than or equal to $p-1$. In the next Theorem, we generalize both the above results by proving:

\begin{th46}
Let $p$ be an odd prime and $G$ be a finite $p$-group. If the nilpotency class of $G$ is at most $p-1$, then $\e(G\wedge G)\mid \e(G)$. In particular, $\e(\M)|\e(G)$.
\end{th46}

The second condition in the next Theorem generalizes the definition of powerful $2$-groups for odd primes and were considered by L. E. Wilson in \cite{LEW}, and the first condition includes the class of groups considered by Arganbright in \cite{DEA}, and it also includes the class of potent $p$-groups considered by J. Gonzalez-Sanchez and  A. Jaikin-Zapirain in \cite{SZ}. Thus as a corollary of the next Theorem, we obtain the well-known result that (\ref{E1}) holds for powerful $p$ groups (\cite{L.M}) and potent $p$ groups (\cite{P.M.4}). 

\begin{th49}
Let $p$ be an odd prime and $G$ be a finite $p$-group satisfying either of the conditions below:
\begin{itemize}
\item[(i)] $\gamma_m(G)\subset G^{p}$ for some  $m$ with $2\leq m\leq p-1$.
\item[(ii)] $\gamma_p(G)\subset G^{p^2}$.
\end{itemize}
 Then $\e(G\wedge G)\mid \e(G)$, hence $\e(\M)\mid \e(G)$.
\end{th49}

As a corollary to the next Theorem, we show that proving Schurs conjecture for regular groups is equivalent to proving it for groups $G$ with $\e(G)=p$.

 \begin{th410}
The following statements are equivalent:
\begin{itemize}
\item[$(i)$] $\e(G\wedge G)\mid \e(G)$ for all regular $p$-groups $G$.
\item[$(ii)$] $\e(G\wedge G)\mid \e(G)$ for all groups $G$ of exponent $p$.
\end{itemize}
 \end{th410}

In section 5, we give bounds on the exponent of $\M$ that depend on nilpotency class. Ellis in \cite{G.E} showed that if $G$ is a group with nilpotency class $c$, then $\e(\M)\mid ( \e(G))^{\ceil{\frac{c}{2}}}$. P. Moravec in \cite{P.M.1} improved this bound by showing that if $d$ is the derived length of $G$, then $\e(\M)\mid ( \e(G))^{2(d-1)}$. In the next Theorem, we improve both the above bounds given in \cite{G.E} and \cite{P.M.1}.

\begin{th53}
Let $G$ be a group with nilpotency class $c>2$ and let  $n= \ceil{\log_{2}(\frac{c+1}{3})}$. If $\e(G)$ is odd, then $\e(G\wedge G) \mid (\e(G))^n$. In particular, $\e(\M)\mid (\e(G))^n$.

\end{th53}

In Theorem 1.1 of \cite{S2}, N. Sambonet improved all the bounds obtained by various authors by proving that $\e(\M)\mid (\e(G))^m$, where $m=\floor{\log_{p-1} c}+1$. We improve the bound in given in \cite{S2} by proving, 

\begin{th56}
Let $p$ be an odd prime and G be a finite $p$-group of class $c$ and let $n=\ceil{\log_{p-1}c}$. If $c\neq 1$, then $\e{(G \wedge G)} \mid \e(G)^n$. In particular, $\e(\M)\mid \e(G)^n$.
\end{th56}

For a solvable group of derived length $d$, the author of \cite{S1} proves that $\e(\M)\mid (\e(G))^d$, when $\e(G)$ is odd, and $\e(\M)\mid 2^{d-1}(\e(G))^d$, when $\e(G)$ is even. Using our techniques, we obtain the following generalization of Theorem A of \cite{S1}, which is one of their main results. 

\begin{th63}
Let $G$ be a solvable group of derived length $d$. 
\begin{itemize}
\item[$(i)$] If $\e(G)$ is odd, then $\e(G\otimes G)\mid (\e(G))^d$. In particular, $\e(\M)\mid (\e(G))^d$.
\item[$(ii)$] If $\e(G)$ is even, then $\e(G\otimes G)\mid 2^{d-1}(\e(G))^d$. In particular, $\e(\M)\mid 2^{d-1}(\e(G))^d$.
\end{itemize}
\end{th63}

\section{Preparatory Results}

R. Brown and J.-L. Loday introduced the nonabelian tensor product $G\otimes H$ for a pair of groups $G$ and $H$ in \cite{BL1} and \cite{BL2} in the context of an application in homotopy theory, extending the ideas of J.H.C. Whitehead in \cite{W}. A special case, the nonabelian tensor square, already appeared in the work of R.K. Dennis in \cite{RKD}. The non-abelian tensor product of groups is defined for a pair of groups that act on each other provided the actions satisfy the compatibility conditions of Definition \ref{D:1.0} below. Note that we write conjugation on the left, so $^gg'=gg'g^{-1}$ for $g,g'\in G$ and
$^gg'\cdot g'^{-1}=[g,g']$ for the commutator of $g$ and $g'$.

\begin{definition}\label{D:1.0}
Let $G$ and $H$ be groups that act on themselves by conjugation and each of which acts on the other. The mutual actions are said to be compatible if
\begin{equation}
^{^h g}h'=\; ^{hgh^{-1}}h' \;and\; ^{^g h}g'=\ ^{ghg^{-1}}g' \;\mbox{for \;all}\; g,g'\in G, h,h'\in H.
\end{equation}
\end{definition}

\begin{definition}\label{D:1.2}
Let $G$ be a group that acts on itself by conjugation, then the nonabelian tensor square $G\otimes G$ is the group generated by the symbols $g\otimes h$ for $g,h\in G$  with relations
\begin{equation}\label{eq:1.1.1}
gg'\otimes h=(^gg'\otimes \;^gh)(g\otimes h),  
\end{equation}
\begin{equation}\label{eq:1.1.2}
g\otimes hh'=(g\otimes h)(^hg\otimes \;^hh'),   
\end{equation}
\noindent for all $g,g',h,h'\in G $.
\end{definition}

There exists a homomorphism $\kappa : G\otimes G \rightarrow G^{\prime}$ sending $g\otimes h$ to $[g,h]$. Let $\nabla (G)$ denote the subgroup of $G\otimes G$ generated by the elements $x\otimes x$ for $x\in G$. The exterior square of $G$ is defined as $G\wedge G= (G\otimes G)/\nabla (G)$. We get an induced homomorphism, which we also denote as $\kappa$, where $\kappa : G\wedge G \rightarrow G^{\prime}$.  We set $M(G)$ as the kernel of this induced homomorphism, which is also known as the Schur multiplier of $G$. It has been shown in \cite{M} that $M(G)\cong H_{2}(G, \mathbb{Z})$, the second homology group of $G$.

We can find the following results  in \cite{BJR} and Proposition 3 of \cite{V}.

\begin{prop}\label{P:2}
\begin{itemize}
\item[(i)] There are homomorphisms of groups $\lambda : G \otimes H \rightarrow G,\ \lambda': G \otimes H \rightarrow H$ such that $\lambda(g \otimes h) = g^hg^{-1},\ \lambda'(g \otimes h)=\; ^ghh^{-1}$.
\item[(ii)] The crossed module rules hold for $\lambda$ and $\lambda'$, that is,
\begin{align*}
\lambda(^gt) &=\ g(\lambda(t))g^{-1},\\
tt_1t^{-1} &=\ ^{\lambda(t)}t_1,
\end{align*}
for all $t,t_1 \in G \otimes H, g \in G$ (and similarly for $\lambda'$).
\item[(iii)] $\lambda(t) \otimes h =\ t^ht^{-1},\ g \otimes \lambda'(t) =\  ^gtt^{-1}$, and thus $\lambda(t) \otimes \lambda'(t_1) = [t,t_1]$ for all $t,t_1 \in G \otimes H, g \in G, h \in H$. Hence $G$ acts trivially on $Ker \lambda'$ and $H$ acts trivially on $Ker \lambda$.
\end{itemize}
In particular, the following relations hold for $g,g_1 \in G$ and $h,h_1 \in H$ :
\begin{itemize}
\item[(iv)]\begin{equation}\label{eq:1.2.1}
^g(g^{-1} \otimes h) = (g \otimes h)^{-1} =\ ^h(g \otimes h^{-1}). 
\end{equation} 

\item[(v)]\begin{equation}\label{eq:1.2.2}
(g \otimes h)(g_1 \otimes h_1)(g \otimes h)^{-1} = (^{[g,h]}g_1 \otimes \ ^{[g,h]}h_1). 
\end{equation}

\item[(vi)]\begin{equation}\label{eq:1.2.3}
[g,h] \otimes h_1 = (g \otimes h)^{h_1} (g\otimes h)^{-1}.
\end{equation}

\item[(vii)]\begin{equation}\label{eq:1.2.4}
 g_1 \otimes\ [g,h] =\ ^{g_1}(g \otimes h) (g\otimes h)^{-1}. 
\end{equation}

\item[(viii)]\begin{equation}\label{eq:1.2.5}
 [g\otimes h, g_1 \otimes h_1] =  [g,h] \otimes [g_1,h_1].
\end{equation}
Moreover, 
\begin{equation}\label{eq:1.2.6}
(g_1 \otimes h_1)(g_2 \otimes h_2) = ([g_1,h_1] \otimes [g_2,h_2])(g_2 \otimes h_2)(g_1 \otimes h_1). 
\end{equation}

\end{itemize}
\end{prop}
Before we state the next result, let us define what we mean by $weight$ of a non-identity element in a nilpotent group $G$. 

\begin{definition}
An element $g \in G\setminus\{1\}$ is said to have $weight$ $n$ if $g \in \gamma_n(G)$ and $g \notin \gamma_{n+1}(G)$. It is denoted by $w(g)$.
\end{definition}
Note that in this paper all the commutators are considered to be right normed and $[g,h] = ghg^{-1}h^{-1}$.
As in \cite{L.G}, we have
\begin{align}\label{eq:1.5.1}
(gh)^n \equiv &\ \prod_{r = n-1}^{1} [h _r, g]^{{n}\choose{r+1}} g^n h^n\ \mbox{mod\ M},
\end{align}
where $M$ is generated by commutators in $g$ and $h$ of weight atleast $2$ in $g$.

\section{Nilpotency Class 5}

In \cite{R} and \cite{EL}, the authors give an isomorphism between the nonabelian tensor square of $G$ and the subgroup $[G, G^{\phi}]$ of $\gamma(G)$.  We use this isomorphism in the following lemma.

\begin{lemma}\label{L:3}
Let $G$ be a group of nilpotency class $c \leq 5$. Then the following hold for $g, g_1, g_2 ,h, h_1, h_2\in G$.
\begin{itemize}

\item[(i)]$ (g \otimes h) = 1$, when $w(g) + w(h) \geq 7$.

\item[(ii)]$ [g_1\otimes h_1, g_2 \otimes h_2 ] = 1$,  when $w(g_1) +w(g_2)+w(h_1)+w(h_2) \geq 7$. In particular, $ ^{[g_1,h_1]}(g_2 \otimes h_2) = (g_2\otimes h_2)$.

\item[(iii)] $\gamma_3(G)$ acts trivially on $g \wedge h$, where $w(g) + w(h) \geq 4$.

\item[(iv)]  $\gamma_2(G)$ acts trivially on $g \wedge h$, where $w(g) + w(h) \geq 5$.

\item[(v)] $((g \wedge [g,h])(g \wedge h))^n = ([g,h] \wedge [g,[g,h]])^{{n}\choose{2}}(g \wedge [g,h])^n(g \wedge h)^n.$
\end{itemize}
\end{lemma}
\begin{proof}

\begin{itemize}

\item[$(i)$] Consider the map $\psi : G \otimes G \rightarrow [G, G^{\phi}]$  defined by $\psi(g \otimes h) = [g, h^{\phi}] $, where $G^{\phi}$ is an isomorphic copy of G.
Note that, $w(g) + w(h) \geq 7$  gives $w(g) + w(h^{\phi}) \geq 7$.
Hence $[g, h^{\phi}] \in \gamma_7(\eta(G)) = 1$, which yields $g \otimes h = 1$.

\item[$(ii)$] Again consider the map $\psi$ as in ($i$). 
Since $ w(g_1) +w(h_1) +w(g_2)+w(h_2) \geq 7$, $w(g_1) +w(h_1^{\phi}) +w(g_2)+w(h_2^{\phi}) \geq 7$. Hence
\begin{align*}
\psi([g_1 \otimes h_1, g_2\otimes h_2]) = [[g_1,h_1^{\phi}],[g_2,h_2^{\phi}]] \in \gamma_7(\eta(G)) = 1.
\end{align*}
Therefore $[g_1 \otimes h_1, g_2 \otimes h_2] =1$, giving us the required result.
Moreover, \begin{align*}
^{[g_1,h_1]}(g_2 \otimes h_2) =&\ (g_1\otimes h_1)(g_2\otimes h_2)(g_1\otimes h_1)^{-1} \ \mbox{by\ }\eqref{eq:1.2.2}\\
&=\ [g_1 \otimes h_1, g_2\otimes h_2](g_2 \otimes h_2) \\
&=\ g_2 \otimes h_2.\\
\end{align*}

\item[$(iii)$] For $a \in \gamma_3(G), w(a) \geq 3$. Therefore, $w(a) + w(g) + w(h) \geq 7$ and the result follows from $(ii)$.
\item[$(iv)$] Follows as in $(iii)$.

\item[$(v)$]Using \eqref{eq:1.5.1}, we have
$((g \wedge [g,h])(g \wedge h))^n \equiv \prod_{r = n-1}^{1} [(g \wedge h)_r, g \wedge [g,h]]^{{n}\choose{r+1}} (g \wedge [g,h])^n(g \wedge h)^n$ mod $M$, where $M$ is generated by commutators in $g \wedge [g,h]$ and $g\wedge h$ of weight atleast $2$ in $g \wedge [g,h]$. Now we will show that $M = 1$. Towards that,
\begin{align*}
[[g \wedge h, g \wedge [g,h]], g \wedge [g,h]] &= [[g,h] \wedge [g,[g,h]], g \wedge [g,h]] \ \mbox{by\ }\eqref{eq:1.2.5} \\
&= 1 \ (by\ (ii)). 
\end{align*}

Also note that, 
\begin{align*}
[g \wedge h, [g \wedge h, g \wedge [g,h]]] &= [g \wedge h, [g,h] \wedge [g,[g,h]]] \ \mbox{by\ }\eqref{eq:1.2.5}\\
&= 1\ (by\ (ii)).
\end{align*}
Therefore the product terminates at $r = 1$. The result now follows by applying \eqref{eq:1.2.5} to the only remaining term in the product.

\end{itemize}
\end{proof}
Let us recall the following combinatorial identity which will be used frequently: ${{n}\choose{r}} + {{n}\choose{r-1}} = {{n+1}\choose{r}}$. Note that ${{n}\choose{r}} := 0$, when $ r >n$. The next lemma gives us information about the action of $g^n$ on $g\wedge h$, and will be crucially used in the expansion of $g^n\wedge h$.

\begin{lemma}\label{L:0}
Let $G$ be a nilpotent group of class $5$. Then for $g,h \in G$, we have
$^{g^n}(g \wedge h) = (g \wedge [g,g,g,g,h])^{{n}\choose{4}}(g \wedge [g,g,g,h])^{{n}\choose{3}}(g \wedge [g,g,h])^{{n}\choose{2}}(g \wedge [g,h])^{n}(g \wedge h)$ for all $n\in \mathbb{N}$.
\end{lemma}
\begin{proof}
We proceed by induction on $n$. The claim clearly holds for $n = 1$ and we prove for $n$. Towards that end,
\begin{align*}
^{g^{n}}(g \wedge h) =&\ ^g{^{g^{n-1}}(g \wedge h)}.
\end{align*}
Applying induction hypothesis and then distributing the action of $g$ onto the individual terms yield
\begin{align*}
^{g^{n}}(g \wedge h) =&\ ^g(g \wedge [g,g,g,g,h])^{{n-1}\choose{4}}\ {^g(g \wedge [g,g,g,h])}^{{n-1}\choose{3}}\ {^g(g \wedge [g,g,h])}^{{n-1}\choose{2}}\\&\ {^g(g \wedge [g,h])}^{n-1}\ {^g(g \wedge h)}.
\end{align*}
Now distribute the action inside each exterior using the identity $^g(g \wedge a) = (g \wedge [g,a]a)$, and then expand each exterior using \eqref{eq:1.1.2}. Note that the actions that arise vanish by \eqref{eq:1.2.2}. Furthermore, by applying Lemma \ref{L:3} ($ii$) to each of the terms with a power, we can distribute the powers. The first term becomes trivial by Lemma \ref{L:3} ($i$). Now combining the powers of similar terms gives the desired result.
\end{proof}

Now we come to an important technical Theorem which is needed in the proof of the main theorem of this section.
\begin{theorem}\label{T:1}
Let $G$ be a nilpotent group of class $5$. Then for $g,h \in G$, we have
\begin{align*}
g^n\wedge h = &\ ([g,h] \wedge [g,g,g,h])^{{n}\choose{4}} ([g,h] \wedge [g,g,h])^{{n}\choose{3}} (g \wedge [g,g,g,g,h])^{{n}\choose{5}}\\
&(g \wedge [g,g,g,h])^{{n}\choose{4}}(g \wedge [g,g,h])^{{n}\choose{3}}(g\wedge [g,h])^{{n}\choose{2}}(g\wedge h)^n 
\end{align*}
for all $n\in \mathbb{N}$.
\end{theorem}
\begin{proof}
We proceed by induction on $n$. It can be easily seen that the statement holds for $n = 1$. We will prove for a general $n$. Towards that, we have
\begin{align*}
g^n \wedge h =\ ^{g^{n-1}}(g\wedge h)(g^{n-1} \wedge h).
\end{align*}
Now applying Lemma \ref{L:0} to the first term and using induction hypothesis for the second term yields
\begin{align*}
g^n \wedge h =&\ (g \wedge [g,g,g,g,h])^{{n-1}\choose{4}}(g \wedge [g,g,g,h])^{{n-1}\choose{3}}(g \wedge [g,g,h])^{{n-1}\choose{2}}\\&\ (g \wedge [g,h])^{n-1}(g \wedge h)([g,h] \wedge [g,g,g,h])^{{n-1}\choose{4}} ([g,h] \wedge [g,g,h])^{{n-1}\choose{3}}\\&\ (g \wedge [g,g,g,g,h])^{{n-1}\choose{5}}(g \wedge [g,g,g,h])^{{n-1}\choose{4}}(g \wedge [g,g,h])^{{n-1}\choose{3}}\\&\ (g\wedge [g,h])^{{n-1}\choose{2}}(g\wedge h)^{n-1}.
\end{align*}
By Lemma \ref{L:3} ($ii$), $g \wedge h$ commutes with all the terms except $(g \wedge [g,g,h])$ and $(g \wedge [g,h])$. Similarly, every term except $g \wedge h$ commutes with one another. Now collecting similar terms using the formula $ab = [a,b]ba$ yields the result we seek.
\end{proof}

\begin{lemma}\label{L:4}
Let $G$ be a group of nilpotency class $5$. Then the following hold for $g,h,a,b,c \in G$ and $n \in \mathbb{N}$.
\begin{itemize}
\item[(i)] If $w(h) \geq 3$, then $(g \wedge h)^n =\ (g \wedge h^n)$. Also, $(h \wedge g)^n =\ (h^n \wedge g)$.
\item[(ii)]$[a,b]^n\wedge g = ([a,b]\wedge [[a,b],g])^{{n}\choose{2}}([a,b]\wedge g)^n$ and $g \wedge [a,b]^n=\ ([a,b] \wedge [g, [a,b]])^{{n}\choose{2}}(g \wedge [a,b])^n$.
\item[(iii)] $(g \wedge h)([g,h]\wedge [a,b])^n = ([g,h] \wedge [[g,h],[a,b]])^n([g,h] \wedge [a,b])^n(g \wedge h) $.
\item[(iv)] $(a \wedge b)([g,h]\wedge [a,b])^n = ([a,b]\wedge [[g, h], [a,b]])^n([g,h] \wedge [a,b])^n(a \wedge b) $.
\end{itemize}
\end{lemma}

\begin{proof}
\begin{itemize}
\item[($i$)] We proceed by induction on $n$. Note that the claim holds for $n = 1$. Now we will prove it for $n$. Write $h^n= hh^{n-1}$. Expanding using \eqref{eq:1.1.2} and then applying induction hypothesis yields 
\begin{align*}
g \wedge h^n  &=\ (g \wedge h)(^{h}(g\wedge h))^{n-1}.
\end{align*}
Applying Lemma \ref{L:3} $(iii)$, we obtain the result.

\item[($ii$)] Applying Theorem \ref{T:1} to $[a,b]^n\wedge g$ and then using Lemma \ref{L:3} ($i$) yields
\begin{align*}
[a,b]^n\wedge g =&\ ([a,b]\wedge [[a,b],g])^{{n}\choose{2}}([a,b]\wedge g)^n.
\end{align*}
Now taking inverse on both sides to the first identity and observing that $([a,b]\wedge [[a,b],g])^{-1} = ([a, b]\wedge [g, [a, b]])$ will give the second identity.

\item[($iii$)] By Proposition \ref{P:2} ($iii$), we have
\begin{align*}
(g \wedge h)([g,h]\wedge [a,b])^n(g \wedge h)^{-1}([g,h]\wedge [a,b])^{-n} =&\ ([g,h] \wedge [[g,h],[a,b]]^{n}).
\end{align*}
Now rearranging terms and applying $(i)$ to $([g,h] \wedge [[g,h],[a,b]]^{n})$ yields the result we seek.

\item[($iv$)] The proof follows mutatis mutandis the proof of $(iii)$.
\end{itemize}
\end{proof}

Before going to the next lemma, we will state a few commutator identities that we will use frequently in the sequel.

\begin{lemma}\label{L:10}
For $g,g_1,h,h_1 \in G$, we have
\begin{itemize}
\item[(i)]\begin{equation}\label{eq:1.11.1}
[gg_1,h] =\ ^g[g_1,h][g,h].
\end{equation}
\item[(ii)]\begin{equation}\label{eq:1.11.2}
[g,hh_1] = [g,h]^h[g,h_1].
\end{equation}
\item[(iii)]\begin{equation}\label{eq:1.11.3}
gh = [g,h]hg.
\end{equation}
\item[(iv)]\begin{equation}\label{eq:1.11.4}
^gh = [g,h]h. 
\end{equation}
\item[(v)]\begin{equation*}
 [g,abcd] = [g,a]^a[g,b]^{ab}[g,c]^{abc}[g,d].
\end{equation*}
\begin{equation}\label{eq:1.11.5}
g \wedge abcd = (g \wedge a)^a(g \wedge b)^{ab}(g \wedge c)^{abc}(g \wedge d). 
\end{equation}
\item[(vi)]\begin{equation*}
[abcd,g] =\ ^{abc}[d,g]^{ab}[c,g]^{a}[b,g][a,g].
\end{equation*}
\begin{equation}\label{eq:1.11.6}
abcd \wedge g =\ ^{abc}(d \wedge g)^{ab}(c \wedge g)^a(b \wedge g)(a \wedge g).
\end{equation}
\end{itemize}
\end{lemma}

In the next lemma, we give conditions which enable us to pull powers outside a commutator.

\begin{lemma}\label{L:5}
Let $G$ be a  group of class $5$. Then the following hold for $g, g_1, h \in G$:
\begin{itemize}
\item[(i)] If\ $w(g) + w(h) \geq 6,\ then\ ^g h = h$.
\item[(ii)] If\ $w(h) \geq 3,\ then\ [g, h^n]= [g ,h]^n$.
\item[(iii)] If $ w(h) \geq 4,\ then\ [g^n, h]= [g, h]^n$.
\item[(iv)] $ [g_1,[g, h]^n] = [g_1,[g, h]]^n[[g, h], [g_1, g, h]]^{{n}\choose{2}}$.
\end{itemize}
\end{lemma}
\begin{proof}
\item[($i$)] The result follows since $ ^gh =[g, h]h$ and $[g, h] \in \gamma_6(G).$
\item[($ii$)] We proceed by induction on $n$. The claim being true for $n =1$, we prove for $n$. Write $h^n$ in $[g, h^n]$ as $hh^{n-1}$ and then expand  using \eqref{eq:1.11.2}.
Now by $(i)$, the action in the expansion becomes trivial and the result follows by applying the induction hypothesis.

\item[($iii$)]We use induction on $n$. The claim is true for $n = 1$ and we prove for $n$. Write $g^n$ in $[g^n, h]$ as $gg^{n-1}$ and then expand using \eqref{eq:1.11.1}. 
Again, by $(i)$ the action in the expansion becomes trivial and the result follows by applying the induction hypothesis.

\item[($iv$)] Again we use induction. The claim being true for $n =1$, we will prove for $n$. Expanding using \eqref{eq:1.11.2} and then applying induction hypothesis yields
\begin{align*}
[g_1,[g, h]^n] &= [g_1,[g, h]]^{[g, h]}([g_1,[g, h]]^{n-1}[[g, h], [g_1, g, h]]^{{n-1}\choose{2}}).
\end{align*}
Now $^{[g, h]}[g_1, [g, h]] = [[g, h], [g_1, g, h]][g_1, g, h]$ by \eqref{eq:1.11.4}  and $[[g, h], [g_1, g, h]]$ commutes with $[g, h], [g_1, g, h]$ by $(i)$. Then we get the result by combining powers of similar terms.

\end{proof} 

The next lemma provides some useful information about the commutator $[g_2 \wedge h_2, (g_1 \wedge h_1)^n]$.

\begin{lemma}\label{L:6}
Let $G$ be a $p$-group of class $5$. Then for $g_i, h_i \in G,$ where $i \in \{1,2\}$, we have 
\begin{align*}
(g_2 \wedge h_2)(g_1 \wedge h_1)^n =& ([g_1,h_1]\wedge [[g_2, h_2], [g_1,h_1]])^{{n}\choose{2}}([g_2,h_2] \wedge [g_1,h_1])^n\\
& (g_1 \wedge h_1)^n(g_2\wedge h_2),
\end{align*}
for all $n\in \mathbb{N}$.
\end{lemma}
\begin{proof}
 By Proposition \ref{P:2}, we have
\begin{align*}
(g_2 \wedge h_2)(g_1 \wedge h_1)^n(g_2 \wedge h_2)^{-1}(g_1 \wedge h_1)^{-n} =\ ([g_2,h_2]\wedge [g_1,h_1]^n) .
\end{align*} 
Now applying Lemma \ref{L:4} $(ii)$ to $([g_2,h_2]\wedge [g_1,h_1]^n)$ and rearranging the terms yields the result.
\end{proof}

Let $G$ be a group and let $a,b\in G$. The next Theorem is similar to the expansion of $(ab)^n$, where $a,b\in G$ for any group $G$. It will be crucially used in the proof of the main theorem.

\begin{theorem}\label{T:2}
Let $G$ be a group of class $5$. Then for $g_i, h_i \in G,$ where $i \in \{1,2\}$, we have
\begin{align*}
((g_1 \wedge h_1)(g_2 \wedge h_2))^n =&\ ([g_2,h_2]\wedge [[g_2,h_2],[g_1,h_1]])^{{n}\choose{3}}([g_1,h_1]\wedge [[g_2, h_2], [g_1,h_1]])^{2{{n}\choose{3}}+{{n}\choose{2}}}\\&\ ([g_2,h_2] \wedge [g_1,h_1])^{{n}\choose{2}} (g_1 \wedge h_1)^n(g_2\wedge h_2)^n,
\end{align*}
for all $n\in \mathbb{N}$.
\end{theorem}

\begin{proof}
We proceed by induction on $n$. The claim holds for $n = 1$. 
Now we will prove the claim for $n$. Towards that, writing $((g_1 \wedge h_1)(g_2 \wedge h_2))^n$ as $((g_1 \wedge h_1)(g_2 \wedge h_2))((g_1 \wedge h_1)(g_2 \wedge h_2))^{n-1}$ and applying induction hypothesis, we obtain
\begin{align*}
((g_1 \wedge h_1)(g_2 \wedge h_2))^n =& (g_1 \wedge h_1)(g_2 \wedge h_2)([g_2,h_2] \wedge [[g_2,h_2],[g_1,h_1]])^{{n-1}\choose{3}}\\&([g_1,h_1]\wedge [[g_2, h_2], [g_1,h_1]])^{2{{n-1}\choose{3}}+{{n-1}\choose{2}}}([g_2,h_2] \wedge [g_1,h_1])^{{n-1}\choose{2}}\\
& (g_1 \wedge h_1)^{n-1}(g_2\wedge h_2)^{n-1}.
\end{align*}
Note that the third and fourth terms in the above equation are central by Lemma \ref{L:3} ($ii$). Therefore, we can commute the second and first terms with them. Now by using Lemma \ref{L:4} ($iii$) to $(g_2\wedge h_2)([g_2,h_2] \wedge [g_1,h_1])^{{n-1}\choose{2}}$, we have

\begin{align*}
((g_1 \wedge h_1)(g_2 \wedge h_2))^n =&\ ([g_2,h_2] \wedge [[g_2,h_2],[g_1,h_1]])^{{n-1}\choose{3}}\\&\ ([g_1,h_1]\wedge [g_2, h_2], [g_1,h_1]])^{2{{n-1}\choose{3}}+{{n-1}\choose{2}}}\\&\ (g_1 \wedge h_1)([g_2,h_2] \wedge [[g_2,h_2],[g_1,h_1]])^{{n-1}\choose{2}}\\&\ ([g_2,h_2] \wedge [g_1,h_1])^{{n-1}\choose{2}}(g_2 \wedge h_2)(g_1 \wedge h_1)^{n-1}(g_2\wedge h_2)^{n-1}.
\end{align*}
Applying Lemma \ref{L:6} to $(g_2 \wedge h_2)(g_1 \wedge h_1)^{n-1}$ and combining powers of similar terms, we obtain
\begin{align*}
((g_1 \wedge h_1)(g_2 \wedge h_2))^n =&\ ([g_2,h_2] \wedge [[g_2, h_2],[g_1,h_1]])^{{{n}\choose{3}}}\\&\ ([g_1,h_1]\wedge [[g_2, h_2], [g_1,h_1]])^{2{{n}\choose{3}}}(g_1 \wedge h_1)\\&\ ([g_2,h_2] \wedge [g_1,h_1])^{{{n}\choose{2}}}(g_1 \wedge h_1)^{n-1}(g_2\wedge h_2)^{n}.
\end{align*}
Now by using Lemma \ref{L:4} ($iv$) to $(g_1\wedge h_1)([g_2, h_2]\wedge [g_1, h_1])^{{n}\choose{2}}$ and combining powers of similar terms  the result follows.

\end{proof}

The next lemma is used to prove the main theorem for a $3$-group of class 5.

\begin{lemma}\label{L:7}
Let $G$ be a group of nilpotency class $5$. Then the following hold for $g,g_1,g_2 ,h ,h_1,h_2\in G$:
\begin{itemize}
\item[(i)] $[g^3,h] = [[g,h],[g,g,h]][g,g,g,h][g,g,h]^3[g,h]^3$. In particular, $[g,h^3] = [g^3,h] = [g,h]^3$, whenever $w(g), w(h) \geq 2$.
\item[(ii)] $g^3 \wedge [g^3,h] = (g \wedge [[g,h],[g,g,h]])^3([g,g,h] \wedge [g^3,[g,h]^3])^3([g,h]\wedge [g^3, [g, h]])^3 (g \wedge [g, g, g, g, h])^3(g \wedge [g,g,g,h])^3(g^3 \wedge [g,g,h])^3(g^3 \wedge [g,h])^3$.
\item[(iii)] $[g^3,[g,h]^3] = [g^3,g,h]^3[[g,h],[g^3,g,h]]^3$.
\item[(iv)]$ [g^3,g^3,h] = [[g,h],[g^3,g,h]]^3[g,g,g,g,h]^3[g^3,g,g,h]^3[g^3,g,h]^3.$
\item[(v)] $[g^3,g^3,g^3,h] = [g^3,g^3,g,g,h]^3[g^3,g^3,g,h]^3.$
\item[(vi)] $[[g_2^3,h_2],[g_1^3,h_1]] = [[g_2,h_2],[g_1,g_1,h_1]]^9[[g_2,g_2,h_2],[g_1,h_1]]^9[[g_2,h_2],[g_1,h_1]]^9$.
\item[(vii)] $[g^3,g,g,h] = [g,g,g,g,h]^3[g,g,g,h]^3$.
\item[(viii)]$[g^3,g,h] = [g,g,g,g,h][g,g,g,h]^3[g,g,h]^3$.
\end{itemize}
\end{lemma}
\begin{proof}
\begin{itemize}
\item[($i$)] 
Expanding using \eqref{eq:1.11.1} twice we have, $[g^3,h] =\ ^{g^2} [g,h] ^g [g,h] [g,h]$. Applying \eqref{eq:1.11.4} on the first two terms, we obtain $[g^3,h] = [g^2,g,h][g,h]$\\ $[g,g,h][g,h]^2.$ Now expand the first term using \eqref{eq:1.11.1} and flip second and third terms using \eqref{eq:1.11.3}. Note that $[[g,h],[g,g,h]]$ is central. Again applying \eqref{eq:1.11.4} on the first term and further by collecting similar terms we arrive at the desired result.

\item[($ii$)] Applying (i) and expanding using \eqref{eq:1.11.5}, we have
\begin{align*}
g^3 \wedge [g^3,h] =&\ (g^3 \wedge [[g,h],[g,g,h]])^{[[g,h],[g,g,h]]}(g^3 \wedge [g,g,g,h])\\&\ ^{[[g,h],[g,g,h]][g,g,g,h]}(g^3 \wedge [g,g,h]^3)\\&\ ^{[[g,h],[g,g,h]][g,g,g,h][g,g,h]^3}(g^3 \wedge [g,h]^3).
\end{align*}
Now applying Lemma \ref{L:3} ($ii$) to the last three terms, we obtain
\begin{align*}
g^3 \wedge [g^3,h] =&\ (g^3 \wedge [[g,h],[g,g,h]])(g^3 \wedge [g,g,g,h])\\&(g^3 \wedge [g,g,h]^3)^{[g,g,h]^3}(g^3 \wedge [g,h]^3)\\
=&\ (g^3 \wedge [[g,h],[g,g,h]])(g^3 \wedge [g,g,g,h])\\&(g^3 \wedge [g,g,h]^3)([g,g,h]^3 \wedge [g^3,[g,h]^3])(g^3 \wedge [g,h]^3),\\ \mbox{by\ Proposition}\ \ref{P:2}\ (iii).
\end{align*}

Apply Theorem \ref{T:1} to the first two terms. Further applying Lemma \ref{L:4} ($i$) to $(g^3 \wedge [g,g,h]^3)$, $([g,g,h]^3 \wedge [g^3,[g,h]^3])$ and Lemma \ref{L:4} ($ii$) to the last term respectively yields
\begin{align*}
g^3 \wedge [g^3,h] =&\  (g \wedge [[g,h],[g,g,h]])^3(g \wedge [g,g,g,g,h])^3(g \wedge [g,g,g,h])^3\\&(g^3 \wedge [g,g,h])^3([g,g,h] \wedge [g^3,[g,h]^3])^3([g,h]\wedge [g^3, [g,h]])^3\\& (g^3 \wedge [g,h])^3.
\end{align*}
Note that all these terms commute with each other by Lemma \ref{L:3} ($ii$) and rearranging these terms gives the result. 
\item[($iii$)] 
Expand $[g^3,[g,h]^3]$ using \eqref{eq:1.11.2} twice, to obtain $[g^3,[g,h]^3]= [g^3,[g,h]]\\^{[g,h]}[g^3,[g,h]]^{[g,h]^2}[g^3,[g,h]]$. Now apply \eqref{eq:1.11.4} on the last two terms and then further expand $[[g,h]^2,[g^3,[g,h]]]$ using \eqref{eq:1.11.1} considering $[g^3,[g,h]]$ as $h$. Noting that the action becomes trivial by Lemma \ref{L:5} ($i$) and then collecting similar terms, we arrive at the result we seek.
\item[($iv$)]
Apply ($i$) to $[g^3,h]$ in $[g^3,[g^3,h]]$ and then expand using \eqref{eq:1.11.5}. Observe that the first term vanishes and the actions become trivial by using Lemma \ref{L:5} ($i$). Now apply Lemma  \ref{L:5} ($iii$),($ii$) to the first term and second term respectively to see that the power $3$ comes out of the commutators. Further applying ($iii$) to the last term gives the result.

\item[($v$)]
Apply ($iv$) to $[g^3,g^3,h]$ in $[g^3,[g^3,g^3,h]]$ and then expand using \eqref{eq:1.11.5}. Note that the first two terms vanish and the actions become trivial by using Lemma \ref{L:5} ($i$). Further apply  Lemma \ref{L:5} ($ii$) to the first and second term to see that the power $3$ comes out of the commutators to give the desired result.

\item[($vi$)] 
Apply ($i$) to $[g_1^3,h_1]$ in $[[g_2^3,h_2],[g_1^3,h_1]]$ and then expand using \eqref{eq:1.11.5}. Observe that the first two terms vanish and the actions become trivial by using Lemma \ref{L:5} ($i$). Further apply ($i$) to the remaining terms to see that the power $3$ on $[g_1,g_1,h_1]$ and $[g_1,h_1]$ comes out of their corresponding commutators. Now apply ($i$) to $[g_2^3, h_2]$ inside both the commutators and expand both the terms using \eqref{eq:1.11.6}. Note that the terms of weight greater than $5$ becomes trivial. Now applying ($i$) to each of the terms in the resulting expression gives the result we seek. 
\item[$(vii)$] and $(viii)$ follows from ($i$).
\end{itemize}
\end{proof}

Before proceeding to the next result, let us look at the following combinatorial properties:

 $(i)\ 3^{n} \mid {3^{n}\choose{m}}$, where $(3,m) =1$.

$(ii)\ 3^{n-1} \mid {3^{n}\choose{3}}$.

The next lemma gives bounds on the exponent of some specific elements of $G\wedge G$.

\begin{lemma}\label{L:8}
Let $G$ be a $3$-group of class less than or equal to $5$ and of exponent $3^n$. Then the following hold for $g,h,g_1,g_2,h_1,h_2 \in G$:
\begin{itemize}
\item[(i)] $(g^3 \wedge [g^3,h])^{{3^{n-1}}\choose{2}} = 1$.
\item[(ii)] $(g^3 \wedge [g^3,g^3,h])^{{{3^{n-1}}}\choose{3}} = 1$.
\item[(iii)] $ (g^3 \wedge [g^3,g^3,g^3,h])^{{{3^{n-1}}}\choose{4}} = 1$.
\item[(iv)] $ (g^3 \wedge [g^3,g^3,g^3,g^3,h])^{{{3^{n-1}}}\choose{5}} = 1$.
\item[(v)] $([g^3,h] \wedge [g^3,g^3,h])^{{{3^{n-1}}}\choose{3}} = 1$.
\item[(vi)] $([g^3,h] \wedge [g^3,g^3,g^3,h])^{{3^{n-1}}\choose{4}} = 1$.
\item[(vii)]$([g_2^3,h_2] \wedge [g_1^3,h_1])^{{3^{n-1}}\choose{2}} = 1$.
\item[(viii)] $([g_1^3,h_1]\wedge [[g_2^3, h_2], [g_1^3,h_1]])^{2{{3^{n-1}}\choose{3}}+{{3^{n-1}}\choose{2}}} = 1$.
\item[(ix)] $([g_2^3,h_2] \wedge [[g_2^3,h_2],[g_1^3,h_1]])^{{3^{n-1}}\choose{3}} = 1.$
\end{itemize}
\end{lemma}
\begin{proof}
\begin{itemize}
\item[($i$)]  From Lemma \ref{L:7} ($ii$) we have,
\begin{align*}
g^3 \wedge [g^3,h] =& (g \wedge [[g,h],[g,g,h]])^3([g,g,h] \wedge [g^3,[g,h]^3])^3([g,h] \wedge [g^3,[g,h]])^3\\& (g \wedge [g,g,g,g,h])^3(g \wedge [g,g,g,h])^3(g^3 \wedge [g,g,h])^3(g^3 \wedge [g,h])^3.
\end{align*}
Note that all the terms in the above expression commute with one another by Lemma \ref{L:3} ($ii$). Therefore
\begin{align*}
(g^3 \wedge [g^3,h])^{3^{n-1}(\frac{3^{n-1}-1}{2})} =&\ \{(g \wedge [[g,h],[g,g,h]])^{3^n}([g,h] \wedge [g^3,[g,h]])^{3^n}\\
&\ ([g,g,h] \wedge [g^3,[g,h]^3])^{3^n}(g \wedge [g,g,g,g,h])^{3^n}\\
&\ (g \wedge [g,g,g,h])^{3^n}(g^3 \wedge [g,g,h])^{3^n}(g^3 \wedge [g,h])^{3^n}\}^{\frac{3^{n-1}-1}{2}}.
\end{align*}
Now applying Lemma \ref{L:4} ($ii$) to the last term and  Lemma \ref{L:4} ($i$) to all the other terms in the above expansion, we can see that $(g^3 \wedge [g^3,h])^{{3^{n-1}}\choose{2}}$ is trivial.

\item[($ii$)] Apply Lemma \ref{L:7} ($iv$) to $g^3 \wedge [g^3,g^3,h]$ and then expand using \eqref{eq:1.11.5}. Note that the actions become trivial by Lemma (\ref{L:3}) ($ii$). Now applying Lemma (\ref{L:4}) ($i$) to each of the term yields,
\begin{align*}
g^3 \wedge [g^3,g^3,h] =& (g^3 \wedge [[g,h],[g^3,g,h]])^3(g^3 \wedge [g,g,g,g,h])^3(g^3 \wedge [g^3,g,g,h])^3\\ & (g^3 \wedge [g^3,g,h])^3.
\end{align*}
Since $[[g,h],[g^3,g,h]]$ is central, we can take the power $3$ outside. Now applying Lemma (\ref{L:7}) ($vii$),($viii$) to the last two terms respectively, we obtain

\begin{align*}
g^3 \wedge [g^3,g^3,h] =&\  (g \wedge [[g,h],[g^3,g,h]])^9(g^3 \wedge [g,g,g,g,h])^3(g^3 \wedge [g,g,g,g,h]^3\\&\ [g,g,g,h]^3)^3(g^3 \wedge [g,g,g,g,h][g,g,g,h]^3[g,g,h]^3)^3.
\end{align*}
Expand the last two terms using \eqref{eq:1.11.5} and note that the actions that arise are trivial by Lemma (\ref{L:3}) ($ii$). Now taking out powers using Lemma (\ref{L:4}) ($i$) yields
\begin{align*}
g^3 \wedge [g^3,g^3,h] =&\ (g \wedge [[g,h],[g^3,g,h]])^9(g^3 \wedge [g,g,g,g,h])^3\{(g^3 \wedge [g,g,g,g,h])^3\\&\ (g^3 \wedge [g,g,g,h])^3\}^3\{(g^3 \wedge [g,g,g,g,h])(g^3\wedge [g,g,g,h])^3\\&\ (g^3 \wedge [g,g,h])^3\}^3.
\end{align*}
Now since all the terms commute with each other by Lemma \ref{L:3} ($ii$), we have
\begin{align*}
g^3 \wedge [g^3,g^3,h] =&\ (g \wedge [[g,h],[g^3,g,h]])^9(g^3 \wedge [g,g,g,g,h])^{15}(g^3 \wedge [g,g,g,h])^{18}\\&\ (g^3 \wedge [g,g,h])^9.
\end{align*}
As $[g,g,g,g,h]$ is central, the power $3$ can be taken outside $g^3 \wedge [g,g,g,g,h]$.
Since $3^{n-2} \mid {{3^{n-1}}\choose{3}}$, we have ${{3^{n-1}}\choose{3}} =\ 3^{n-2}k_0$, for some integer $k_0$.
Therefore
\begin{align*}
(g^3 \wedge [g^3,g^3,h])^{{3^{n-1}}\choose{3}} =&\ \{(g \wedge [[g,h],[g^3,g,h]])^{3^n}(g \wedge [g,g,g,g,h])^{3^n\times 5}\\&\ (g^3 \wedge [g,g,g,h])^{3^n \times 2}(g^3 \wedge [g,g,h])^{3^n}\}^{k_0}.
\end{align*}
Now the result follows by applying Lemma \ref{L:4} ($i$) to each of the terms.
\item[($iii$)] By Theorem \ref{T:1} and using Lemma \ref{L:3} ($i$), we have
\begin{align*}
(g^3 \wedge [g^3,g^3,g^3,h])^{{3^{n-1}}\choose{4}} =&\ \{(g \wedge [g,g^3,g^3,g^3,h])^3(g \wedge [g^3,g^3,g^3,h])^3\}^{{3^{n-1}}\choose{4}}.
\end{align*}
Since $3^{n-1} \mid {{3^{n-1}}\choose{4}}$, we have ${{3^{n-1}}\choose{4}} =\ 3^{n-1}k_1$, for some integer $k_1$. Then,
\begin{align*}
(g^3 \wedge [g^3,g^3,g^3,h])^{{3^{n-1}}\choose{4}} =&\ \{(g \wedge [g,g^3,g^3,g^3,h])^{3^n}(g \wedge [g^3,g^3,g^3,h])^{3^n}\}^{k_1}.
\end{align*}
Now applying Lemma \ref{L:4} ($i$), we can see that $3^n$ can be taken inside. Hence $(g^3 \wedge [g^3,g^3,g^3,h])^{{3^{n-1}}\choose{4}}$ becomes trivial.

\item[($iv$)] Consider $(g^3 \wedge [g^3,g^3,g^3,g^3,h])^{{{3^{n-1}}}\choose{5}}$. As  $\ 3^{n-1} \mid {{{3^{n-1}}\choose{5}}}$ and $[g^3,g^3,g^3,g^3,h]$ is central, the power $3^{n-1}$ can be taken inside the exterior, making it trivial.

\item[($v$)]Apply Lemma \ref{L:7} ($iv$) to $([g^3,h] \wedge [ g^3,g^3,h])$ and expand using \eqref{eq:1.1.2} three times. The first two terms vanish and the actions become trivial by Lemma \ref{L:3} ($i$) and ($ii$) respectively. Now by Lemma \ref{L:4} ($i$) we have,
\begin{align*}
[g^3,h] \wedge [ g^3,g^3,h] =& \ ([g^3,h] \wedge [g^3,g,g,h])^3([g^3,h] \wedge [g^3,g,h])^3.
\end{align*}
Apply Lemma \ref{L:7} ($vii$),($viii$) to the first and second terms respectively. Now expand using \eqref{eq:1.1.2} once for the first term and twice for the second term  as in \eqref{eq:1.11.5}. Note that the actions vanish by Lemma (\ref{L:3}) ($ii$). Since all the terms commute, the power $3$ on the product of first two terms and last three terms can be distributed uniformly to all the terms. Now the first and third terms become trivial by Lemma (\ref{L:3}) ($i$) and after combining similar terms, we obtain
\begin{align*}
[g^3,h] \wedge [ g^3,g^3,h] =&\ ([g^3,h] \wedge [g,g,g,h]^3)^6 ([g^3,h] \wedge [g,g,h]^3)^3.
\end{align*}
Now by Lemma (\ref{L:4}) ($i$), the power $3$ can be taken out from both these terms to yield us
\begin{align*}
[g^3,h] \wedge [ g^3,g^3,h] =& ([g^3,h] \wedge [g,g,g,h])^{18}([g^3,h] \wedge [g,g,h])^9.
\end{align*}

Since $3^{n-2} \mid {{3^{n-1}}\choose{3}}$, we have ${{3^{n-1}}\choose{3}} = 3^{n-2}k_2$, for some integer $k_2$. Then,
\begin{align*}
([g^3,h] \wedge [ g^3,g^3,h])^{{3^{n-1}}\choose{3}} =& \{([g^3,h] \wedge [g,g,g,h])^{18}([g^3,h] \wedge [g,g,h])^9\}^{3^{n-2}k_2}\\
=& \{([g^3,h] \wedge [g,g,g,h])^{2\times 3^n}([g^3,h] \wedge [g,g,h])^{3^{n}}\}^{k_2}
\end{align*}
Now by applying Lemma \ref{L:4} ($i$) on both the terms, we obtain the desired result.

\item[($vi$)] Apply Lemma \ref{L:7} ($v$) to $([g^3,h] \wedge [g^3,g^3,g^3,h])$ and expand using \eqref{eq:1.1.2}. Then noting that the action becomes trivial by Lemma \ref{L:3} ($ii$), yields
\begin{align*}
([g^3,h] \wedge [g^3,g^3,g^3,h]) =&\ ([g^3,h] \wedge [g^3,g^3,g,g,h]^3) ([g^3,h] \wedge [g^3,g^3,g,h]^3).
\end{align*}
Observe that, first term becomes trivial by Lemma \ref{L:3} ($i$). Further using Lemma \ref{L:4} ($i$), we obtain $([g^3,h] \wedge [g^3,g^3,g,h]^3) =\ ([g^3,h] \wedge [g^3,g^3,g^3,h])$. Therefore $([g^3,h] \wedge [g^3,g^3,g^3,h])^{{3^{n-1}}\choose{4}} =\ ([g^3,h] \wedge [g^3,g^3,g,h])^{3^n k_1}$ becomes trivial after applying Lemma \ref{L:4} ($i$) as in $(iii)$.

\item[($vii$)] Apply Lemma \ref{L:7} ($i$) to $[g_1^3,h_1]$ in $[g_2^3,h_2] \wedge [g_1^3,h_1]$ and then expand using \eqref{eq:1.11.5}. The first term vanishes and the actions become trivial by Lemma \ref{L:3} ($i$),($ii$) respectively. Again apply Lemma \ref{L:7} ($i$) to $[g_2^3,h_2]$ in $([g_2^3,h_2] \wedge [g_1,g_1,g_1,h_1])$ and then expand using \eqref{eq:1.11.6}. Note that the second, third and fouth terms vanish and the action on the first term becomes trivial by Lemma (\ref{L:3}) ($i$),($ii$) respectively. We are left with the following three terms,
\begin{align*}
[g_2^3,h_2] \wedge [g_1^3,h_1] =&\ ([g_2^3,h_2] \wedge [g_1,g_1,g_1,h_1]) ([g_2^3,h_2]\wedge [g_1,g_1,h_1]^3)\\&\ ([g_2^3,h_2]\wedge [g_1,h_1]^3).
\end{align*}
Since $[[g_2,h_2]^3,[g_1,g_1,g_1,h_1]] =1$, the power $3$ can be taken outside the first term. Also by Lemma \ref{L:4} ($i$), the power can be taken ouside the second term. Now applying Lemma \ref{L:4} ($ii$) to the last term, we obtain
\begin{align*}
[g_2^3,h_2] \wedge [g_1^3,h_1] =&\ ([g_2,h_2] \wedge [g_1,g_1,g_1,h_1])^3 ([g_2^3,h_2]\wedge [g_1,g_1,h_1])^3\\&\ ([g_2^3,h_2]\wedge [g_1,h_1])^3([g_1,h_1]\wedge [[g_2^3,h_2],[g_1,h_1]])^3.
\end{align*}

Note that these terms commute with each other by Lemma \ref{L:3} ($ii$). Therefore,
$([g_2^3,h_2] \wedge [g_1^3,h_1])^{{3^{n-1}}\choose{2}} = 1$, by applying Lemma \ref{L:4} ($ii$) for the third term and ($i$) for the rest of the terms in the above expression.
\end{itemize}

\item[($viii$)] Since $3^{n-2} \mid 2{{3^{n-1}}\choose{3}} +  {{3^{n-1}}\choose{2}}$, we have $2{{3^{n-1}}\choose{3}} +  {{3^{n-1}}\choose{2}} = 3^{n-2}k_3$ for some integer $k_3$. 

From Lemma \ref{L:7} ($vi$), we obtain

\begin{align*}
([g_1^3,h_1]\wedge [[g_2^3, h_2], [g_1^3,h_1]])^{2{{3^{n-1}}\choose{3}} +  {{{3^{n-1}}}\choose{2}}} =& ([g_1^3, h_1]\wedge [[g_2,h_2],[g_1,g_1,h_1]]^9\\&[[g_2,g_2,h_2],[g_1,h_1]]^9 [[g_2,h_2],[g_1,h_1]]^9)^{3^{n-2}k_3}.
\end{align*}
Because the commutators in the exterior commute, the power $9$ can be taken outside and the result follows.
Similarly ($ix$) can be proved.

\end{proof}

The next lemma provides a bound on the exponent of $G^3\wedge G$.

\begin{lemma}\label{L:9}
Let $G$ be a $3$-group of class less than or equal to $5$ and exponent $3^n$. Then the exponent of the image of $G^3 \wedge G$ in $G \wedge G$ divides $3^{n-1}$.
\end{lemma}
\begin{proof}
Towards proving the claim, first we do the same for a simple exterior, $g^3 \wedge h$, for $g,h \in G$ and then for a product. Finally, showing the claim holds for $g_1^3g_2^3 \wedge h$, where $g_1,g_2 \in G$, completes the proof.
By Theorem \ref{T:1}, we have
\begin{align*}
(g^3)^{3^{n-1}} \wedge h =&\ ([g^3,h] \wedge [g^3,g^3,g^3,h])^{{3^{n-1}}\choose{4}} ([g^3,h] \wedge [g^3,g^3,h])^{{{3^{n-1}}}\choose{3}}\\ 
 &\ (g^3 \wedge [g^3,g^3,g^3,g^3,h])^{{{3^{n-1}}}\choose{5}}
(g^3 \wedge [g^3,g^3,g^3,h])^{{{3^{n-1}}}\choose{4}}\\
&\ (g^3 \wedge [g^3,g^3,h])^{{{3^{n-1}}}\choose{3}}(g^3\wedge [g^3,h])^{{{3^{n-1}}}\choose{2}}(g^3\wedge h)^{3^{n-1}}.
\end{align*}
Now applying Lemma \ref{L:8}, we obtain $(g^3 \wedge h)^{3^{n-1}} = 1$, for $g, h \in G$.  

Let $g_1 ^3, g_2^3 \in G^3 $ and $h_1,h_2 \in G$. Then from Theorem \ref{T:2}, we have
 \begin{align*}
((g_1^3 \wedge h_1)(g_2^3 \wedge h_2))^{3^{n-1}} =&\ ([g_2^3,h_2]\wedge [[g_2^3,h_2],[g_1^3,h_1]])^{{3^{n-1}}\choose{3}}\\&\ ([g_1^3,h_1]\wedge [[g_2^3,h_2],[g_1^3,h_1]])^{2{{3^{n-1}}\choose{3}}+{{3^{n-1}}\choose{2}}}\\&\ ([g_2^3,h_2] \wedge [g_1^3,h_1])^{{3^{n-1}}\choose{2}}(g_1^3 \wedge h_1)^{3^{n-1}}(g_2^3\wedge h_2)^{3^{n-1}}.
\end{align*}
Again applying Lemma \ref{L:8}, yields $((g_1^3 \wedge h_1)(g_2^3 \wedge h_2))^{3^{n-1}} = 1$. Also, for $g_1^3,g_2^3 \in G^3$ and $h \in G$,
\begin{align*}
(g_1^3g_2^3 \wedge h)^{3^{n-1}} =& (^{g_1^3}(g_2^3 \wedge h)(g_1^3 \wedge h))^{3^{n-1}}\\
=& 1.
\end{align*}
Therefore, we have $\e(G^3 \wedge G) \mid 3^{n-1}$.
\end{proof}

\begin{lemma} \label{L:30} Let $G$ be a nilpotent group of class $c$ and $a, b, c, d, e\in G$. 
\begin{itemize}
\item[($i)$] If $w(a)+w(b)+w(c)\ge c+1$, then $^a[b, c] = [b, c]$.
\item[($ii$)] If $w(a)+w(b)+w(c)+w(d)\ge c+1$, then $[a, b][c, d] = [c, d][a, b]$.
\item[($iii$)] If $w(a)+w(b)+w(c)\ge c+1$, then $[ab, c] = [a, c] [b, c]$.
\item[($iv$)] If $w(a)+w(b)+w(c)+w(d)\ge c+1$, then $[a, bc, d] = [a, b, d][a, c, d]$.
\item[($v$)] If $w(a)+w(b)+w(c)+w(d)+w(e)\ge c+1$, then $[a, b, cd, e] = [a, b, c, e] [a, b, d, e]$.
\end{itemize}
\end{lemma}

\begin{proof}
\begin{itemize}
\item[($i$)] Applying \eqref{eq:1.11.4} to $^a[b, c]$ gives $(i)$.
\item[($ii$)] Applying \eqref{eq:1.11.3} to $[a, b][c, d]$ gives $(ii)$.
\item[($iii$)] Expanding $[ab, c]$ using \eqref{eq:1.11.1} yields $[ab, c] = \;^ a[b, c][a, c]$. The action becomes trivial by $(i)$ and both the terms commute by $(ii)$ giving the result.
\item[($iv$)] Using \eqref{eq:1.11.1}, \eqref{eq:1.11.4} we have $[bc, d] = [b, c, d][c, d][b, d]$. Then expanding twice $[a, [bc, d]]=[a,[b, c, d][c, d][b, d]]$ using \eqref{eq:1.1.2} gives, 
\begin{equation*}[a, bc, d] = [a, b, c, d] ^{[b, c, d]}[a, c, d] ^{[b, c, d][c, d]}[a, b, d]
\end{equation*}
Note that the actions becomes trivial by $(i)$ and $[a,b,c,d]=1$, proving the result. Similarly $(v)$ can be proved.
\end{itemize}
\end{proof}

\begin{lemma}\label{L:12}
Let $G$ be a group of nilpotenct class $5$. Then the following hold for $g,h \in G$.
\begin{itemize}
\item[($i$)] $[gh,gh,gh,gh,h] = [h,h,h,g,h][h,h,g,g,h][h,g,h,g,h][g,h,h,g,h]$\\ $[h,g,h,g,h][g,h,g,g,h][g,g,h,g,h][g,g,g,g,h]$.
\item[($ii$)] $[h,h,gh,gh,h] = [h,h,h,g,h][h,h,g,g,h]$.
\item[($iii$)] $[h,gh,h,gh,h] = [h,h,h,g,h][h,g,h,g,h]$.
\item[($iv$)] $[gh,h,h,gh,h] = [h,h,h,g,h][g,h,h,g,h]$.
\item[($v$)] $[h,gh,gh,gh,h] = [h,h,h,g,h][h,h,g,g,h][h,g,h,g,h][h,g,g,g,h]$.
\item[($vi$)] $[gh,h,gh,gh,h] = [h,h,h,g,h][h,h,g,g,h][g,h,h,g,h][g,h,g,g,h]$.
\item[($vii$)] $[gh,gh,h,gh,h] = [h,h,h,g,h][g,h,h,g,h][h,g,h,g,h] [g,g,h,g,h]$. 
\end{itemize}
\end{lemma}
\begin{proof}
Expanding using \eqref{eq:1.11.1} gives $[gh, h] =[g, h]$. Using Lemma \ref{L:30} $(iii)$ we have, 
\begin{equation*}[gh, gh, gh, g, h] = [g, gh, gh, g, h] [h, gh, gh, g, h].
\end{equation*}
 Applying Lemma \ref{L:30} $(iv)$ to terms on the right hand side of the above equation yields,
\begin{align*}&[g, gh, gh, g, h] = [g, g, gh, g, h] [g, h, gh, g, h], \\&[h, gh, gh, g, h] = [h, g, gh, g, h] [h, h, gh, g, h].
\end{align*}
Applying Lemma \ref{L:30} $(v)$ to each terms on the right of the equalities above, we obtain
 \begin{align*} &[g, g, gh, g, h] = [g, g, g, g, h] [g, g, h, g, h], \\
 &[g, h, gh, g, h] = [g, h, g, g, h] [g, h, h, g, h], \\
 &[h, g, gh, g, h] = [h, g, g, g, h [h, g, h, g, h], \\
 &[h, h, gh, g, h] = [h, h, g, g, h [h, h, h, g, h].
 \end{align*} Also, all the eight terms commute by Lemma \ref{L:30} $(ii)$. This proves $(i)$ and $(ii)$ - $(vii)$ can be proved similarly.
\end{proof}

The next Theorem proves the conjecture for a $5$-group of class $5$.

\begin{theorem}
Let $G$ be a $5$-group of class $5$. Then, $\e(G \wedge G) \mid \e(G)$.
\end{theorem}

\begin{proof}
Let $\e(G)=5^n$ for some $n \in \mathbb{N}$. By Theorem \ref{T:2}, its enough to prove that $(g \wedge h)^n = 1,$ for all $g,h \in G$. From Theorem \ref{T:1} and using Lemma \ref{L:4} (i) and (ii), we obtain
\begin{align*}
1 = g^{5^n} \wedge h = (g \wedge [g,g,g,g,h])^{{5^n}\choose{5}}(g \wedge h)^{5^n}\ (*)
\end{align*}

We will show that $(g \wedge [g,g,g,g,h])^{{5^n}\choose{5}} = 1$. Towards that,
replacing $g$ by $gh$ in $(*)$ yields
\begin{align*}
1 =&\ (gh \wedge [gh,gh,gh,gh,h])^{{5^n}\choose{5}}(gh \wedge h)^{5^n}.
\end{align*}
Now, using Lemma \ref{L:12} ($i$) and expanding, we have
\begin{align*}
1 =& \{^g \{(h \wedge [h,h,h,g,h])(h \wedge [h,h,g,g,h])(h \wedge [h,g,h,g,h])(h \wedge [h,g,g,g,h])\\& (h \wedge [g,h,h,g,h])(h \wedge [g,h,g,g,h])(h \wedge [g,g,h,g,h])(h \wedge [g,g,g,g,h])\}\\& (g \wedge [h,h,h,g,h])(g \wedge [h,h,g,g,h])(g \wedge [h,g,h,g,h])(g \wedge [h,g,g,g,h])\\& (g \wedge [g,h,h,g,h])(g \wedge [g,h,g,g,h])(g \wedge [g,g,h,g,h])(g \wedge [g,g,g,g,h])\}^{{5^n}\choose{5}}\\&(g \wedge h)^{5^n}.
\end{align*}
Note that the action, once distributed onto each term, vanishes and the terms commute with one another. Further using $(*)$, we obtain
\begin{align*}
 1 =& \{(h \wedge [h,h,h,g,h])(h \wedge [h,h,g,g,h])(h \wedge [h,g,h,g,h])(h \wedge [h,g,g,g,h])\\& (h \wedge [g,h,h,g,h])(h \wedge [g,h,g,g,h])(h \wedge [g,g,h,g,h])(h \wedge [g,g,g,g,h])\}\\& (g \wedge [h,h,h,g,h])(g \wedge [h,h,g,g,h])(g \wedge [h,g,h,g,h])(g \wedge [h,g,g,g,h])\\& (g \wedge [g,h,h,g,h])(g \wedge [g,h,g,g,h])(g \wedge [g,g,h,g,h])\}^{{5^n}\choose{5}}.\ (**)
\end{align*}
Replacing $g$ by $gh$ in the above expression and further expanding using Lemma \ref{L:12} yields
\begin{align*}
1 =& \{(h \wedge [h,h,h,g,h])^{15}(h \wedge [h,h,g,g,h])^7(h \wedge [h,g,h,g,h])^7(h \wedge [h,g,g,g,h])^3\\& (h \wedge [g,h,h,g,h])^7(h \wedge [g,h,g,g,h])^3(h \wedge [g,g,h,g,h])^3(h \wedge [g,g,g,g,h])\\&
(g \wedge [h,h,h,g,h])^7(g \wedge [h,h,g,g,h])^3(g \wedge [h,g,h,g,h])^3(g \wedge [h,g,g,g,h])\\& (g \wedge [g,h,h,g,h])^3(g \wedge [g,h,g,g,h])(g \wedge [g,g,h,g,h])\}^{{5^n}\choose{5}}.
\end{align*}
Now by applying $(**)$ gives
\begin{align*}
1 =& \{(h \wedge [h,h,h,g,h])^{14}(h \wedge [h,h,g,g,h])^6(h \wedge [h,g,h,g,h])^6(h \wedge [h,g,g,g,h])^2\\& (h \wedge [g,h,h,g,h])^6(h \wedge [g,h,g,g,h])^2(h \wedge [g,g,h,g,h])^2(g \wedge [h,h,h,g,h])^6\\&(g \wedge [h,h,g,g,h])^2(g \wedge [h,g,h,g,h])^2(g \wedge [g,h,h,g,h])^3\}^{{5^n}\choose{5}}.\ (***)
\end{align*}
Thus we have,
\begin{align*}
&\{(h \wedge [h,g,g,g,h])^2(h \wedge [g,h,g,g,h])^2(h \wedge [g,g,h,g,h])^2(g \wedge [h,h,g,g,h])^2\\&(g \wedge [h,g,h,g,h])^2(g \wedge [g,h,h,g,h])^3\}^{{5^n}\choose{5}} = \{((h \wedge [h,h,h,g,h])^{-4}(h \wedge [h,h,g,g,h])^{-1}\\&(h \wedge [h,g,h,g,h])^{-1}(h\wedge [g,h,h,g,h])^{-1}(g \wedge [h,h,h,g,h])^{-1}\}^{{5^n}\choose{5}}.\ (****)
\end{align*}
Again, replacing $g$ by $gh$ in $(***)$ and expanding yields
\begin{align*}
1 =&  \{(h \wedge [h,h,h,g,h])^{51}(h \wedge [h,h,g,g,h])^{12}(h \wedge [h,g,h,g,h])^{12}(h \wedge [h,g,g,g,h])^2\\& (h \wedge [g,h,h,g,h])^{13}(h \wedge [g,h,g,g,h])^2(h \wedge [g,g,h,g,h])^2(g \wedge [h,h,h,g,h])^{13}\\&(g \wedge [h,h,g,g,h])^2(g \wedge [h,g,h,g,h])^2(g \wedge [g,h,h,g,h])^3\}^{{5^n}\choose{5}}\\
=&  \{(h \wedge [h,h,h,g,h])(h \wedge [h,h,g,g,h])^{2}(h \wedge [h,g,h,g,h])^{2}(h \wedge [h,g,g,g,h])^2\\& (h \wedge [g,h,h,g,h])^{3}(h \wedge [g,h,g,g,h])^2(h \wedge [g,g,h,g,h])^2(g \wedge [h,h,h,g,h])^{3}\\&(g \wedge [h,h,g,g,h])^2(g \wedge [h,g,h,g,h])^2(g \wedge [g,h,h,g,h])^3\}^{{5^n}\choose{5}},\mbox{\ by\ Lemma \ref{L:4}(i)}.
\end{align*}
Now applying $(****)$ yields,
\begin{align*}
1 =& \{(h \wedge [h,h,h,g,h])^2(h \wedge [h,h,g,g,h])(h \wedge [h,g,h,g,h]) (h \wedge [g,h,h,g,h])^{2}\\& (g \wedge [h,h,h,g,h])^2\}^{{5^n}\choose{5}}.\ (*****)
\end{align*}
Further replacing $g$ by $gh$ in $(*****)$, expanding and applying Lemma \ref{L:4} ($i$) gives,
\begin{align*}
1 =& \{(h \wedge [h,h,h,g,h])^3(h \wedge [h,h,g,g,h])(h \wedge [h,g,h,g,h]) (h \wedge [g,h,h,g,h])^{2}\\& (g \wedge [h,h,h,g,h])^2\}^{{5^n}\choose{5}}.
\end{align*}
Thus, comparing the above expression with $(*****)$, we obtain $(h \wedge [h,h,h,g,h])^{{5^n}\choose{5}} = 1$. Hence, $(g \wedge [g,g,g,h,g])^{{5^n}\choose{5}} = 1$. Now $(g \wedge [g,g,g,g,h])^{{5^n}\choose{5}} = (g \wedge [g,g,g,h,g])^{{-1}{{5^n}\choose{5}}}=1$, and the result follows. 
\end{proof}

Finally we come to the main Theorem of this section.

\begin{theorem}\label{T:3}
Let $G$ be a finite $p$-group of nilpotency class 5. If $p$ is odd, then $\e(\M)\mid \e(G)$\end{theorem}
\begin{proof}
The claim holds when $p \geq 5$. We prove for $p = 3$.
Consider the following exact sequence. 
\begin{align*}
 G^3\wedge G \rightarrow G \wedge G \rightarrow \frac{G}{G^3}\wedge \frac{G}{G^3} \rightarrow 1.
\end{align*}
We have, $\e( G \wedge G )\mid \e(im( G^{3}\wedge G))\ \e( \frac{G}{G^{3}}\wedge \frac{G}{G^{3}} )$. Note that $\e(\frac{G}{G^{3}}) = 3$ and hence from \cite{P.M.1}, we have $\e( \frac{G}{G^{3}}\wedge \frac{G}{G^{3}} ) \mid 3$. By Lemma \ref{L:9}, we have $\e(im( G^{3}\wedge G)) \mid 3^{n-1}$ and hence the claim.
\end{proof}

\section{Regular groups, Powerful groups and groups with power-commutator structure}

In this section, we prove that $\e(\M)\mid \e(G)$ for powerful $p$ groups, the class of groups considered by Arganbright in \cite{DEA}, which includes the class of potent $p$ groups, $p$ groups of class at most $p-1$.

The first part of next lemma can be found in \cite{G.E}, we record its proof here as the order of the factors is different in \cite{G.E}.
\begin{lemma}\label{L:4.1}
Let $N, M$ be normal subgroups of a group $G$ and $i$ be a positive integer. If $N$ is abelian, then for $n, n_1\in N$, and $m, m_1\in M$, we have
\begin{itemize}
\item[$(i)$] ${n^i}\otimes m =(n\otimes [n, m])^{{i}\choose{2}}(n\otimes m)^i$.
\item[$(ii)$] ${n^i}\otimes [n_1, m] =(n\otimes [n_1, m])^i$.
\item[$(iii)$] $((n\otimes m)(n_1\otimes m_1))^i =([n_1, m_1]\otimes [n, m])^{{i}\choose{2}}(n\otimes m)^i(n_1\otimes m_1)^i$.
\end{itemize}
\end{lemma}

\begin{proof}
\begin{itemize}
\item[$(i)$] Note that ${n^2}\otimes m =\ ^n(n\otimes m)(n\otimes m)$, and $^n(n\otimes m) =n\otimes [n, m]m$. Thus ${n^2}\otimes m =(n\otimes [n, m]) \ ^{[n, m]}(n\otimes m)(n\otimes m) =(n\otimes [n, m])(n\otimes m)^2$. Let $i>2$, and assume the statement for $i-1$. Then ${n^{i}}\otimes m =\ ^n(n^{i-1}\otimes m)(n\otimes m)= \;^n{(n\otimes [n, m])^{{i-1}\choose{2}}} \ ^n{(n\otimes m)^{i-1}}(n\otimes m)$. Since $N$ is abelian, $[n, [n, m]]=1$. This implies that $\ ^n(n\otimes [n, m]) =n\otimes [n, m]$, and hence $n\otimes [n, m]$ belongs to the center of $N\otimes M$. Therefore 
\begin{align*}
{n^i}\otimes m&=(n\otimes [n, m])^{{i-1}\choose{2}}((n\otimes [n, m])(n\otimes m))^{i-1}(n\otimes m)
\\&=(n\otimes [n, m])^{{i-1}\choose{2}}(n\otimes [n, m])^{i-1}(n\otimes m)^{i-1}(n\otimes m)\\
&=(n\otimes [n, m])^{{i}\choose{2}}(n\otimes m)^i.
\end{align*}

\item[$(ii)$] This follows from $(i)$, since $[n, [n_1, m]]=1$.

\item[$(iii)$] Since $N$ is abelian, the nilpotency class of $N\otimes M$ is at most $2$ (cf. \cite{DLT}). Hence, 
\begin{align*}
((n\otimes m)(n_1\otimes m_1))^i &= [(n_1\otimes m_1), (n\otimes m)]^{{i}\choose{2}}(n\otimes m)^i(n_1\otimes m_1)^i\\
&=([n_1, m_1]\otimes [n, m])^{{i}\choose{2}}(n\otimes m)^i(n_1\otimes m_1)^i. 
\end{align*}
\end{itemize}
\end{proof}

\begin{lemma}\label{L:4.2}
Let $N, M$ be normal subgroups of a group $G$. $N$ be abelian. \begin{itemize}
\item[(i)] If $\e(N)$ is odd, then $\e(N\otimes M)\mid \e(N)$.
\item[(ii)] If $\e(N)$ is even, then $\e(N\otimes M)\mid 2\e(N)$. \end{itemize}
\end{lemma}

\begin{proof}
Let $\e(N) =e$. $n_i\in N$, $m_i \in M$, $i = 1, 2$.
\begin{itemize}
\item[$(i)$] By Lemma \ref{L:4.1} $(i)$, we obtain ${n_1^e}\otimes m_1 =(n_1\otimes [n_1, m_1])^{{e}\choose{2}}(n_1\otimes m_1)^e =(n_1^{{e}\choose{2}}\otimes [n_1, m_1])(n_1\otimes m_1)^e$.
Since $e$ is odd, it follows that $e\mid {{e}\choose{2}}$, whence $(n_1\otimes m_1)^e=1$. By Lemma \ref{L:4.1}$(iii)$, we obtain
$((n_1\otimes m_1)(n_2\otimes m_2))^e =([n_2, m_2]\otimes [n_1, m_1])^{{e}\choose{2}}(n_1\otimes m_1)^e(n_2\otimes m_2)^e =1$, which proves $\e(N\otimes M)\mid e$.

\item[$(ii)$] By Lemma \ref{L:4.1} $(i)$ and $(ii)$, we have\\
${n_1^{2e}}\otimes m_1 =(n_1\otimes [n_1, m_1])^{{2e}\choose{2}}(n_1\otimes m_1)^{2e} =(n_1^{{2e}\choose{2}}\otimes [n_1, m_1])(n_1\otimes m_1)^{2e}$.\\
Since $e\mid {{2e}\choose{2}}$, it follows that $(n_1\otimes m_1)^{2e}=1$. By Lemma \ref{L:4.1} $(iii)$ and $(ii)$, we obtain
\begin{align*}
(n_1\otimes m_1)(n_2\otimes m_2))^{2e}=&([n_2, m_2]\otimes [n_1, m_1])^{{2e}\choose{2}}(n_1\otimes m_1)^{2e}(n_2\otimes m_2)^{2e}\\
=&([n_1, m_1]^{{2e}\choose{2}}\otimes [n, m])(n\otimes m)^{2e}(n_1\otimes m_1)^{2e} =1.
\end{align*}
Thus $\e(N\otimes M)\mid 2\e(N)$.
\end{itemize}
\end{proof}

Recall that a $p$-group $G$ is said to be powerful if $\gamma_2(G)\subset G^p$, when $p$ is odd and $\gamma_2(G)\subset G^4$, for $p=2$.

\begin{lemma}\label{L:4.3}
Let $G$ be a finite $p$-group. If $G$ is powerful, then the subgroup $G^p$ of $G$ is the set of all $p$th-powers of elements of $G$ and $G^p$ is powerful.
\end{lemma}

The following Lemma can be found in \cite{G.E.1}.

\begin{lemma}\label{L:4.4}
Let $N, M$ be normal subgroups of a group $G$. If $M\subset N$, then we have the exact sequence
$M\wedge G\rightarrow N\wedge G \rightarrow \frac{N}{M}\wedge \frac{G}{N} \rightarrow 1.$
\end{lemma}

In \cite{L.M}, the authors prove the conjecture for powerful groups. The authors of \cite{MHM} prove that if $N$ is powerfully embedded in $G$, then the $\e(\m(G,N))\mid \e(N)$. In the theorem below, we generalize both these results.

\begin{theorem}\label{th:4.5}
Let $p$ be an odd prime and $N$ be a normal subgroup of a finite $p$-group $G$. If $N$ is powerful, then $exp(N\wedge G)\mid exp(N)$.
\end{theorem}

\begin{proof}
Let $\e(N) =p^e$. We will proceed by induction on $e$. If $\e(N) =p$, then $\gamma_2(N)\subset N^p =1$. So by Lemma \ref{L:4.2} $(i)$, $\e(N\wedge G)\mid \e(N)$. Let $\e(N) =p^e, e>1$. Using Lemma \ref{L:4.4}, consider the exact sequence
$$N^p\wedge G\rightarrow N\wedge G \rightarrow \frac{N}{N^p}\wedge \frac{G}{N^p} \rightarrow 1.$$ Then $\e(N\wedge G)\mid \e(im(N^p\wedge G))exp(\frac{N}{N^p}\wedge \frac{G}{N^p})$. By Lemma \ref{L:4.3}, $N^p$ is powerful and $\e(N^p) =p^{e-1}$. Thus by induction hypothesis, we have $\e(N^p\wedge G)\mid \e(N^p)$, and hence $\e(im(N^p\wedge G))\mid \e(N^p)$. Note that $\frac{N}{N^p}$ is a powerful group of exponent $p$, and we have showed above that the theorem holds for powerful groups of exponent $p$. So $\e(\frac{N}{N^p}\wedge \frac{G}{N^p})\mid p$ and hence we have $\e(N\wedge G)\mid p^e$.\\
\end{proof}

P. Moravec in (\cite{P.M.3}) proved that if the nilpotency class is at most $p-2$, then $\e(G\wedge G)\mid \e(G)$. The authors of \cite{H.M.M} prove the conjecture for class $p-1$. In the Theorem below, we extend both these results by proving it for exterior square.

\begin{theorem} \label{T:4.8}
Let $p$ be an odd prime and $G$ be a finite $p$-group. If the nilpotency class of $G$ is at most $p-1$, then $\e(G\wedge G)\mid \e(G)$. In particular, $\e(\M)|\e(G)$.\end{theorem}
\begin{proof}
Let $\e(G) = p^e$ and let $H$ be a covering group of $G$. Then $H$ will be $p^e$-central and $H$ will have nilpotency class $\le p$. Let $a, b\in H$. The group $\langle \ ^b a, a\rangle$ = $\langle [b,a], a\rangle$ has nilpotency class $\le p-1$, hence is $p^e$-abelian. So, $1 = [b, a^{p^e}] = (\ ^b a)^{p^e} (a^{-1})^{p^e} = (\ ^b a a^{-1})^{p^e} = [b, a]^{p^e}$. Also $\gamma_2(H)$ will have nilpotency class $\le p-1$, hence $\gamma_2 (H)$ is $p^e$-abelian. Therefore, $\e(\gamma_2(H))\mid p^e$ and the lemma follows since $G\wedge G\cong \gamma_2(H)$. 
\end{proof}

\begin{definition}\label{D.4.1}
Let $G$ be a finite $p$-group.
\begin{itemize}
\item[$(i)$] We say $G$ satisfies condition $(1)$ if $\gamma_m(G)\subset G^p$, for some $m=2, 3, \dots, p-1$. $G$ is said to be a potent group if $m=p-1$.
\item[$(ii)$] We say $G$ satisfies condition $(2)$ if $\gamma_p(G)\subset G^{p^2}$.
\end{itemize}
\end{definition}

The next Theorem can be found in \cite{SZ} and \cite{LEW}.

\begin{theorem}\label{L:4.7}
Let $G$ be a finite $p$-group. 

\begin{itemize}
\item[$(i)$] If $G$ regular, then $G^p$ is the set of all $pth$ powers of elements of $G$ and $G^p$ is powerful.
\item[$(ii)$] If $G$ satisfies condition $(1)$, then $G^p$ is the set of all $p$th powers of elements of $G$ and $G^p$ is powerful.
\item[$(iii)$] If $G$ satisfies condition $(2)$, $G^p$ is the set of all $p$th powers of elements of $G$ and $G^p$ is powerful.
\end{itemize}

\end{theorem}

Groups satisfying condition ($1$) were studied by D. E. Arganbright in \cite{DEA}. This class includes the class of powerful $p$ groups for $p$ odd and potent $p$ groups. The groups satisfying condition ($2$) were studied by L. E. Wilson in \cite{LEW}. In the theorem below, we show that the conjecture is true for all these classes of groups.

\begin{theorem}\label{T:4.9}
Let $p$ be an odd prime and $G$ be a finite $p$-group satisfying either of the conditions below:
\begin{itemize}
\item[(i)] $\gamma_m(G)\subset G^{p}$ for some  $m$ with $2\leq m\leq p-1$.
\item[(ii)] $\gamma_p(G)\subset G^{p^2}$.
\end{itemize}
 Then $\e(G\wedge G)\mid \e(G)$, hence $\e(\M)\mid \e(G)$.
\end{theorem}

\begin{proof}
\begin{itemize}
\item[(i)] Let $\e(G)=p^e$. If $e=1$, then $\gamma_{p-1}(G)\subset G^p =1$. By Theorem \ref{T:4.8}, we obtain $\e(G\wedge G)\mid \e(G)$. Let $e>1$ and consider the exact sequence $G^p\wedge G\rightarrow G\wedge G \rightarrow \frac{G}{G^p}\wedge \frac{G}{G^p} \rightarrow 1$. We have, $\e(G\wedge G)\mid \e(im(G^p\wedge G)) \e(\frac{G}{G^p}\wedge \frac{G}{G^p})$. By Theorem \ref{L:4.7}, $G^p$ is powerful and $\e(G^p) =p^{e-1}$. Now Theorem \ref{th:4.5} implies that $\e(G^p\wedge G)\mid p^{e-1}$. Since the group $\frac{G}{G^p}$ is nilpotent of class at most $p-2$, by Theorem \ref{T:4.8}$, \e(\frac{G}{G^p}\wedge \frac{G}{G^p})\mid p$. Therefore, $\e(G\wedge G)\mid \e(G^p\wedge G)\e(\frac{G}{G^p}\wedge \frac{G}{G^p})\mid p^{e-1}p =p^e$.

\item[(ii)] Let $\e(G) = p^e$. If $e =2$, then $\gamma_p(G)\subset G^{p^2} =1$. So by Theorem \ref{T:4.8}, $\e(G\wedge G)\mid \e(G)$. Let $e >2$, consider the exact sequence 
$$G^p\wedge G\rightarrow G\wedge G \rightarrow \frac{G}{G^p}\wedge \frac{G}{G^p} \rightarrow 1.$$ We have, $\e(G\wedge G)\mid \e(im(G^p\wedge G))\e(\frac{G}{G^p}\wedge \frac{G}{G^p})$. By Theorem \ref{L:4.7} $G^p$ is powerful and $\e(G^p) =p^{e-1}$. Then by Theorem \ref{th:4.5} $\e(G^p\wedge G)\mid p^{e-1}$. Since $\gamma_p(G)\subset G^{p^2}\subset G^p$, the group $\frac{G}{G^p}$ has nilpotency class $\le p-1$. So by Theorem \ref{T:4.8}, $\e(\frac{G}{G^p}\wedge \frac{G}{G^p})\mid p$. Therefore, $\e(G\wedge G)\mid p^{e-1}p =p^e$. \end{itemize}
\end{proof}

In the next Theorem, we show that to prove the conjecture for regular $p$-groups, it is enough to prove it for groups of exponent $p$. This result also appears in \cite{S1}. But here, we prove it more generally for the exterior square.  
 \begin{theorem}\label{T:4.10}
The following statements are equivalent:
\begin{itemize}
\item[$(i)$] $\e(G\wedge G)\mid \e(G)$ for all regular $p$-groups $G$.
\item[$(ii)$] $\e(G\wedge G)\mid \e(G)$ for all groups $G$ of exponent $p$.
\end{itemize}

 \end{theorem}
 \begin{proof}
 Since groups of exponent $p$ are regular, one direction of the proof is trivial. To see the other direction, let $G$ be a regular group. Suppose $\e(G) = p^e$, $e >1$. Consider the exact sequence $$G^p\wedge G\rightarrow G\wedge G \rightarrow \frac{G}{G^p}\wedge \frac{G}{G^p} \rightarrow 1.$$ We have $\e(G\wedge G)\mid \e(im(G^p\wedge G)) \e(\frac{G}{G^p}\wedge \frac{G}{G^p})$. By Theorem \ref{L:4.7}, $G^p$ is powerful, and $\e(G) = p^{e-1}$. Then by Theorem \ref{th:4.5} $\e(G^p\wedge G)\mid p^{e-1}$, so is $\e(im(G^p\wedge G))$. Since $\frac{G}{G^p}$ has exponent $p$, we will have $\e(\frac{G}{G^p}\wedge \frac{G}{G^p})\mid p$. Therefore $\e(G\wedge G)\mid p^e$.
 \end{proof}

\section{Bounds depending on the nilpotency class}

In the next lemma, we use the isomorphism between the nonabelian tensor square of $G$ and the subgroup $[G, G^{\phi}]$ of $\gamma(G)$ (cf. \cite{EL} and \cite{R}). The next two lemma's are required to prove that $\e(\M)\mid (\e(G))^n$, where $n= \ceil{\log_{2}(\frac{c+1}{3})}$.  

\begin{lemma}\label{L:1}
Let $G$ be a nilpotent group of class $c$. Then for $n = \ceil{\frac{c}{2}}$, the image of $\gamma_{n}(G) \otimes G$ in $G\otimes G$ is abelian.
\end{lemma}

\begin{proof}
Consider the isomorphism $\psi : G\otimes G \rightarrow [G,G^{\phi}]$ defined by $\psi(g\otimes h)= [g,h^{\phi}]$, where $G^{\phi}$ is an isomorphic copy of $G$. Recall that
$[G,G^{\phi}]$ is a subgroup of $\gamma(G)$. By Theorem A of \cite{R}, $\gamma(G)$ is of nilpotency class at most $c+1$.

For $n_1,n_2 \in \gamma_n(G)$ and $g_1,g_2 \in G$, we obtain
\begin{align*}
\psi( [(n_1\otimes g_1),(n_2\otimes g_2)]) &= [[n_1,g_1^{\phi}],[n_2,g_2^{\phi}]] \in \gamma_{2n+2}(\gamma(G)).
\end{align*}
But $\gamma_{2n+2}(\gamma(G)) \leq \gamma_{c+2}(\gamma(G))=1$, giving us the desired result.

\end{proof}

\begin{lemma}\label{L:2}
Let $G$ be a nilpotent group of class $c$ and of odd order. For $n = \ceil{\frac{c}{2}}$, $x \in \gamma_n(G)$ and $ g \in G$, $exp(x \otimes g)$ divides $exp(\gamma_n(G))$.
\end{lemma}
\begin{proof}
For $x \in \gamma_n(G)$ and $g \in G$, $[x, [x,g]] \in \gamma_{2n+1}(G) \leq \gamma_{c+1}(G) = 1$.
Let $t = exp(\gamma_n(G))$. Then by Lemma \ref{L:4.1} ($i$), we have
\begin{align*}
x^t \otimes g &= (x\otimes [x,g]^{\frac{t}{2}(t-1)})(x \otimes g)^t\\
&= (x \otimes g)^t
\end{align*}
and hence the result.
\end{proof}

In \cite{G.E.1}, G. Ellis proves that if the nilpotency class of $G$ is c, then $\e(\M)\mid ( \e(G))^{\ceil{\frac{c}{2}}}$. In \cite{P.M.1}, P. Moravec improves this bound by showing that $\e(\M)\mid ( \e(G))^{2\floor{\log_2 c}}$. In the next Theorem, we improve both these bounds. The cases $c=1,2$ have been excluded as the conjecture is known to be true in those cases.

\begin{theorem}\label{T:5.3}
Let $G$ be a group with nilpotency class $c>2$ and let  $n= \ceil{\log_{2}(\frac{c+1}{3})}$. If $\e(G)$ is odd, then $\e(G\wedge G) \mid (\e(G))^n$. In particular, $\e(\M)\mid (\e(G))^n$.
\end{theorem}

\begin{proof}
The proof is by induction on $n$. Note that, $n \geq \log_2(\frac{c+1}{3})$ if and only if $c \leq (2^n \times 3) - 1$. 
When $n = 1$, the statement is true by Theorem \ref{T:3}. Let us prove for the $n$.
Towards that, consider the following exact sequence which can be obtained from Theorem (3.1) in \cite{V.G}, where $m = \ceil{\frac{c}{2}}$: 

\begin{align*}
 \gamma_m(G)\wedge G \rightarrow G \wedge G \rightarrow \frac{G}{\gamma_m(G) }\wedge \frac{G}{\gamma_m(G)} \rightarrow 1.
\end{align*}

Thus $\e(G \wedge G)\mid \e(im(\gamma_m(G) \wedge G))\e( \frac{G}{\gamma_m(G) }\wedge \frac{G}{\gamma_m(G)})$. By Lemma \ref{L:1} and \ref{L:2}, we obtain that $\e(im(\gamma_m(G) \wedge G)) \mid \e(\gamma_m(G))$. Now $\frac{G}{\gamma_m(G)}$ is of nilpotency class $m-1$ and  $m -1 < (2^{n-1} \times 3) - 1$. Applying the induction hypothesis, we obtain $\e( \frac{G}{\gamma_m(G) }\wedge \frac{G}{\gamma_m(G)} ) \mid \e(G)^{n-1}$ and the result follows.

\end{proof}

We know that if $p> c$, then $\e(\M) \mid \e(G)$. In the next theorem, we show that if we consider primes less than the nilpotency class, but $\ceil{\frac{c}{2}}$ at most $p$, then we get a bound of 2, that is:

\begin{theorem}
Let $p$ be an odd prime and $G$ be a finite $p$-group of nilpotency class $c$. If $m := \ceil{\frac{c}{2}}\leq  p$, then $\e(\M) \mid \e(\gamma_m (G)) \e(\frac{G}{\gamma_m(G)})$. In particular, $\e(\M) \mid (\e(G))^2$.
\end{theorem}

\begin{proof}
Consider the following commutative diagram where $\alpha$ and $\beta$ are the natural commutator maps from the respective domains : 

\begin{equation*}
\xymatrix@+20pt{
&\gamma_m (G)\wedge G\ar@{->}[r]
\ar@{->}^{\alpha}[d]
 &G\wedge G\ar@{->}[r]
\ar@{->}[d]
&\frac{G}{\gamma_m (G)}\wedge \frac{G}{\gamma_m (G)} \ar@{->}[r]
\ar@{->}^{\beta}[d]
&1 \\
1\ \ar@{->}[r] 
&Im(\alpha)\ar@{->}[r]
 &\gamma_2(G)\ar@{->}[r]
&\frac{\gamma_2(G)}{Im(\alpha)}\ar@{->}[r]
&1.
}\end{equation*}

By snake lemma, we obtain $\ker(\alpha) \rightarrow \M \rightarrow \ker(\beta) \rightarrow 1$.
Since $\ker(\beta) \leq \m(\frac{G}{\gamma_m(G)})$,  we have $\e(\M) \mid \e(im(\gamma_m(G)\wedge G))\e(\m(\frac{G}{\gamma_m(G)}))$. Applying lemma \ref{L:1} and \ref{L:2} yields $\e(im(\gamma_m(G) \wedge G)) \mid \e(\gamma_m(G))$. By Corollary 3.4 in \cite{H.M.M}, we have $\e(\m(\frac{G}{\gamma_m(G)})) \mid \e(\frac{G}{\gamma_m(G)})$ and the result follows.

\end{proof}

The next lemma is crucially used for the last theorem of this section.

\begin{lemma}\label{L:5.5}
Let $p$ be an odd prime and $G$ be a finite $p$-group. If $N \unlhd G$ of nilpotency class at most $p-2$, then $\e(N \wedge G) \mid \e(N)$.
\end{lemma}
\begin{proof}
Consider a projective relative central extension, $\delta : N^* \rightarrow G$, associated with a covering group $N^*$ of $N$. We know from \cite{G.E.1} that $[N^*, G] \cong N \wedge G$. Since $N$ is of class atmost $p-2$, $N^*$ is of class atmost $p-1$, and hence regular. Let $\e(N)= t$. Since $N^*$ is $t$ central, $N^*$  is $t$ abelian. Therefore, it is enough to prove that ${(n\ ^g{n^{-1}})}^t = 1$ for $n \in N$ and $g \in G$. Towards that end, 
 \begin{align*}
 {(n\ ^g{n^{-1}})}^t =& n^t{(^g n^{-1})}^t ,\  \mbox{since}\ N^*\ \mbox{is\ $t$ abelian} \\
 =& n^t\ ^g(n^{-t})\\
 =&1,  \mbox{\ since}\ n^t \in Z(N,G) .
  \end{align*}
Hence the result.
\end{proof}

In \cite{S2}, N. Sambonet proves that $\e(\M)\mid (\e(G))^m$, where $m=\floor{\log_{p-1} c}+1$. We improve this bound in the theorem below. Since the conjecture is known to be true for $c=1$, we exclude that case.

\begin{theorem}\label{T:5.6}
Let $p$ be an odd prime and G be a finite $p$-group of class $c$ and let $n=\ceil{\log_{p-1}c}$. If $c\neq 1$, then $\e{(G \wedge G)} \mid \e(G)^n$. In particular, $\e(\M)\mid \e(G)^n$.
\end{theorem}
\begin{proof}
 We proceed by induction on $n$. If $n = 1$, then $c\leq p-1$ and the claim follows by Theorem \ref{T:4.8}. Now assume $n>1$. Set $m= \ceil{\frac{c}{p-1}}+1$ and consider the exact sequence $\gamma_m(G)\wedge G \rightarrow G \wedge G \rightarrow \frac{G}{\gamma_m(G) }\wedge \frac{G}{\gamma_m(G)} \rightarrow 1.$ Hence $\e(G\wedge G)\mid \e(\gamma_m(G) \wedge G) \e( \frac{G}{\gamma_m(G)}\wedge \frac{G}{\gamma_m(G)})$. Since $(p-1)m\ \geq\ (p-1)\frac{c}{p-1} + p - 1 > c$, $\gamma_m(G)$ is of class atmost $p-2$. By Lemma $\ref{L:5.5}$, we obtain $\e(\gamma_m(G) \wedge G) \mid \e(G)$. 
 Since $c\leq (p-1)^n$, we obtain $\frac{c}{p-1}\leq (p-1)^{n-1}$. Hence the nilpotency class of $\frac{G}{\gamma_m(G)}$ is at most $(p-1)^{n-1}$. Using induction hypothesis, we obtain $\e(\frac{G}{\gamma_m(G)}\wedge \frac{G}{\gamma_m(G)}) \mid \e(G)^{n-1}$, and hence the proof.
\end{proof}

\section{Bounds depending on the derived length}

P. Moravec in \cite{P.M.1} proved that the conjecture is true for metabelian groups of exponent $p$. In the theorem below, we prove it for $p$-central metabelian groups.
\begin{theorem}
Let $G$ be a $p$-central metabelian group. Then $\e(G \wedge G) \mid \e(G)$. 
\end{theorem}
\begin{proof}
Let $\e(G) = p^n$, for some integer $n$.
Since $G$ is $p$-central, we have the following commutative diagram where $\alpha$ and $\beta$ are the natural commutator maps from the respective domains : 
\begin{equation*}
\xymatrix@+20pt{
&G^p\wedge G\ar@{->}[r]
\ar@{->}^{\alpha}[d]
 &G\wedge G\ar@{->}[r]
\ar@{->}[d]
&\frac{G}{G^p}\wedge \frac{G}{G^p} \ar@{->}[r]
\ar@{->}^{\beta}[d]
&1 \\
&1\ar@{->}[r]
 &\gamma_2(G)\ar@{->}[r]
&\gamma_2(G)\ar@{->}[r]
&1.
}\end{equation*}

Now Snake Lemma gives the exact sequence,
\begin{align*}
 \ker(\alpha) \rightarrow \M \rightarrow \ker(\beta) \rightarrow 1.
\end{align*}
Since $\ker(\beta) \leq \m(\frac{G}{G^p})$,  we have $\e(\m(G)) \mid \e(im(G^p\wedge G))\e(\m(\frac{G}{G^p}))$. $G$ being $p$-central, $\e(im(G^{p}\wedge G)) \mid p^{n-1}$. Further, $\frac{G}{G^p} $ being a metabelian group of exponent $p$ gives $\e(\m(\frac{G}{G^p})) \mid p$. Hence the claim for a $p$-central metabelian group. 

\end{proof}

The following Lemma can be found in \cite{G.E}.
\begin{lemma}\label{L:6.1}
Let $N$ be a normal subgroup of a group $G$, and $N\subset \gamma_2(G)$. Then the sequence 
$$N\otimes G\rightarrow G\otimes G \rightarrow \frac{G}{N}\otimes \frac{G}{N} \rightarrow 1$$ is exact.\\
\end{lemma}

In \cite{P.M.1}, P. Moravec showed that if $d$ is the derived length of $G$, then $\e(\M)\mid ( \e(G))^{2(d-1)}$. The author of \cite{S1} improved this bound by proving that $\e(\M)\mid (\e(G))^d$, when $\e(G)$ is odd, and $\e(\M)\mid 2^{d-1}(\e(G))^d$, when $\e(G)$ is even. Using our techniques, we obtain the following generalization of Theorem A of \cite{S1}.

\begin{theorem}
Let $G$ be a solvable group of derived length $d$. 
\begin{itemize}
\item[$(i)$] If $\e(G)$ is odd, then $\e(G\otimes G)\mid (\e(G))^d$. In particular, $\e(\M)\mid (\e(G))^d$.
\item[$(ii)$] If $\e(G)$ is even, then $\e(G\otimes G)\mid 2^{d-1}(\e(G))^d$. In particular, $\e(\M)\mid 2^{d-1}(\e(G))^d$.
\end{itemize}
\end{theorem}

\begin{proof}
The proof proceeds by induction on $d$. When $d =1$, $G$ is abelian. So $\e{G\otimes G}\mid \e{G}$. Let $d >1$, then using Lemma \ref{L:6.1} consider the exact sequence below:
$$G^{(d-1)}\otimes G\rightarrow G\otimes G \rightarrow \frac{G}{G^{(d-1)}}\otimes \frac{G}{G^{(d-1)}} \rightarrow 1.$$ We have that $\e(G\otimes G)\mid \e(im(G^{(d-1)}\otimes G))\e(\frac{G}{G^{(d-1}}\otimes \frac{G}{G^{(d-1)}})$.

\begin{itemize}
\item[$(i)$] Let $\e(G)$ be odd. Since $G^{(d-1)}$ is abelian, Lemma \ref{L:4.2} implies that $\e(G^{(d-1)}\otimes G)\mid \e(G^{(d-1)})\mid \e(G)$. Hence by induction hypothesis, $\e(\frac{G}{G^{(d-1)}}\otimes \frac{G}{G^{(d-1)}})\mid (\e(\frac{G}{G^{(d-1)}}))^{d-1}\mid (\e(G))^{d-1}$. Therefore $\e(G\otimes G)\mid (\e(G))^d$.

\item[$(ii)$] Let $\e(G)$ be even. Since $G^{(d-1)}$ is abelian, Lemma \ref{L:4.2} yields $\e(G^{(d-1)}\otimes G)\mid 2\e(G^{(d-1)})\mid 2\e{G}$. Thus by induction hypothesis, $\e(\frac{G}{G^{(d-1)}}\otimes \frac{G}{G^{(d-1)}})\mid 2^{d-2}(\e(\frac{G}{G^{(d-1)}}))^{d-1}\mid 2^{d-2}(\e(G))^{d-1}$. Therefore, $\e(G\otimes G)\mid 2^{d-1}(\e(G))^d$.
\end{itemize}
\end{proof}

In the next lemma, we further generalize the above theorem by considering $N\otimes G$ instead of $G\otimes G$.

\begin{lemma}
Let $N\unlhd G$. Suppose $N$ is solvable of derived length $d$. \begin{itemize}
\item[(i)] If $\e(N)$ is odd, then $\e(N\otimes G)\mid (\e(N))^d$. In particular, $\e(\m(G,N))\mid (\e(N))^d.$
\item[(ii)] If $\e(N)$ is even, then $\e(N\otimes G)\mid 2^d(\e(N))^d$. In particular, $\e(\m(G,N))\mid 2^d(\e(N))^d.$
\end{itemize}
\end{lemma}

\begin{proof}
We prove the Lemma by induction on $d$. Let $d =1$, then the theorem follows from Lemma \ref{L:4.2}. Let $d >1$. The exact sequence $$N^{(d -1)}\otimes G\rightarrow N\otimes G\rightarrow \frac{N}{N^{(d-1)}}\otimes \frac{G}{N^{(d-1)}}\rightarrow 1$$ gives that $\e(N\otimes G)\mid \e(im(N^{(d-1)}\otimes G))\e(\frac{N}{N^{(d-1)}}\otimes \frac{G}{N^{(d-1)}})$.
\begin{itemize}
\item[$(i)$] If $\e(N)$ is odd, then by induction hypothesis $\e(\frac{N}{N^{(d-1)}}\otimes \frac{G}{N^{(d-1)}})\mid \e(\frac{N}{N^{(d-1)}})^{d-1}\mid (\e(N))^{d-1}$. Since $N^{(d-1)}$ is abelian, $\e(N^{(d-1)}\otimes G)\mid \e(N^{(d-1)})\mid \e(N)$, and hence $\e(im(N^{(d-1)}\otimes G))\mid \e(N)$. Therefore, $\e(N\otimes G)\mid (\e(N))^d$.
\item[$(ii)$] If $\e(N)$ is even, then induction hypothesis implies that $\e(\frac{N}{N^{(d-1)}}\otimes \frac{G}{N^{(d-1)}})\mid 2^{d-1}(\e(\frac{N}{N^{(d-1)}}))^{d-1}\mid 2^{d-1}(\e(N))^{d-1}$. Since $N^{(d-1)}$ is abelian, $\e(N^{(d-1)}\otimes G)\mid 2\e(N^{(d-1)})\mid 2\e(N)$. Thus $\e(im(N^{(d-1)}\otimes G))\mid \e(N)$. Therefore $\e(N\otimes G)\mid 2^d(\e(N))^d$.
\end{itemize}
\end{proof}

For a $p$-group of nilpotency class $c$, we compare the bound we have obtained for $\e(\M)$ with some previous bounds. In the following table, we have listed the values of $m$, where $\e(\M) \mid \e(G)^m$, for the bounds given by Ellis, Moravec and Theorem \ref{T:5.3} in this paper.
\begin{center}
\textbf{Table I}
\end{center}
\begin{center}
\begin{tabular}{ | m{7em} | m{7em}| m{7em} | m{7em} | m{7em}|}
\hline
    & Ellis \cite{G.E} &  Moravec \cite{P.M.1}& \\
 \hline
 $c$ & $\ceil{\frac{c}{2}}$ & $2\floor{log_2c}$ & $\ceil{log_2(\frac{c+1}{3})}$\\
 \hline
 3  & 2  & 2  & 1\\
\hline
 4  & 2  & 4  & 1\\
 \hline
 5  & 3  & 4  & 1\\
 \hline
 6  & 3  & 4  & 2\\
  \hline
 11  & 6  & 6  & 2\\
  \hline
 20  & 10  & 8  & 3\\
  \hline
 100  & 50  & 12  & 6\\
  \hline
\end{tabular}
\end{center}

Moravec improves the bound given by Ellis for $c > 11$. It can be seen that the bound obtained in Theorem \ref{T:5.3} improves the other bounds.
In the following table, we consider $p$-groups of nilpotency class $c$ and exponent $p^n$. The bounds $p^m$, where $\e(\M) \mid p^m$, obtained by Moravec, Sambonet and Theorem \ref{T:5.6} are listed.
\begin{center}
\textbf{Table II}
\end{center}
\begin{center}
\begin{tabular}{ | m{1cm} | m{1cm}| m{1cm} | m{7em} | m{8em}| m{5em} |} 
\hline
&  &  & Moravec \cite{P.M.1} & Sambonet \cite{S2}& \\ 
\hline
$c$ & $p$ & $n$ & $p^{k\floor{log_2c}}$ & $p^{n(\floor{log_{p-1}c}+1)}$ &  $p^{n\ceil{log_{p-1}c}}$\\
 \hline
 5 & 3 & 1 & $3^2$ & $3^3$ & $3^3$ \\
 \hline 
 5 & 3 & 2 & $3^8$ & $3^6$ & $3^6$ \\
 \hline
 6 & 7 & 1 & $7^2$ & $7^2$ & $7$ \\
 \hline
 12 & 13 & 2 & $13^{12}$ & $13^4$ & $13^2$ \\
 \hline
 16 & 5 & 1 & $5^{4}$ & $5^{3}$ & $5^{2}$\\
 \hline
 144 & 13 & 1 & $13^{14}$ & $13^3$ & $13^2$ \\
 \hline
\end{tabular}
\end{center}

where $k$ is defined in \cite{P.M.1}. \\
Note that, in Table II even though Moravec has the best bound in the first row, from Table I it is clear that Theorem \ref{T:5.3} gives a better bound for the same case.
\section*{Acknowledgements} In a private communication with the third author, P. Moravec had mentioned that for groups with nilpotency class 5, he could only prove that $\e(\M)\mid (\e(G))^3$. He further mentioned that computer evidence showed that the bound should be 2 instead of 3. Later he himself proved the bound to be 2 in \cite{P.M.5}. We thank P. Moravec for sharing this insight with us.

\section*{References}
\bibliographystyle{amsplain}
\bibliography{Bibliography}
\end{document}